\documentclass[10pt]{article}

\usepackage{latexsym, amssymb, amsmath, amsthm, amscd}
\usepackage[all,cmtip]{xy}
\usepackage{amsmath}
\usepackage{amsthm}
\usepackage{amssymb}
\usepackage{amsfonts}
\usepackage{amsrefs}
\usepackage{color,authblk}
\usepackage{tikz}
\usepackage{amscd} 
\usepackage{comment}
\usepackage{mathrsfs}
\usepackage[all]{xy}
\usepackage{graphicx}
\usepackage{authblk}
\usepackage{hyperref}
\usepackage{fullpage}
\usepackage[all,cmtip]{xy}
\usepackage{amsmath,amscd}
\usepackage{tikz-cd}
\usepackage{tikz}
\usetikzlibrary{matrix}

\usepackage{hyperref}

\numberwithin{equation}{section} \DeclareMathSizes{2}{10}{12}{13}
\newtheorem{thm}{Proposition}[section]
\newtheorem{Thm}[thm]{Theorem}
\newtheorem{rem}[thm]{Remark}

\newtheorem{lem}[thm]{Lemma}
\newtheorem{defn}[thm]{Definition}

\title{Cyclic cohomology of entwining structures}

\author{Mamta Balodi \footnote{Email : mamta.balodi@gmail.com} $\qquad$ Abhishek Banerjee \footnote{Email : abhishekbanerjee1313@gmail.com} \footnote{AB was partially supported by SERB Matrics fellowship MTR/2017/000112}    }

\date{}

\begin{document}

\maketitle 

\medskip

\centerline{\emph{Department of Mathematics, Indian Institute of Science, Bangalore, India.}}

\medskip

\begin{abstract} In this paper, we introduce and study a cyclic cohomology theory $H^\bullet_\lambda(A,C,\psi)$ for an entwining structure $(A,C,\psi)$ over a field $k$. We then provide a complete description of the cocycles and the coboundaries in this theory using entwined traces applied
to dg-entwining structures over $(A,C,\psi)$. We then apply these descriptions to construct a pairing
$
H^m_\lambda(A,C,\psi) \otimes H^n_\lambda(A',C',\psi') \longrightarrow H^{m+n}_\lambda(A \otimes A', C \otimes C', \psi \otimes \psi')
$, 
where $(A,C,\psi)$ and $(A',C',\psi')$ are entwining structures.
\end{abstract}

\medskip
MSC(2010) Subject Classification:    16W30, 16E40

\medskip
Keywords : Entwining structure, cyclic cohomology 
\section{Introduction}

An entwining structure, as introduced by Brzezi\'{n}ski  and Majid \cite{Brz1}, consists of an algebra $A$, a coalgebra $C$ and a map $\psi: C\otimes A\longrightarrow A\otimes C$ satisfying certain conditions. Together, an entwining structure $(A,C,\psi)$ behaves like a bialgebra or more generally, a comodule algebra over a bialgebra, as pointed out by Brzezi\'{n}ski \cite{Brz2}. There is also a well developed theory of modules over entwining structures, with applications to diverse objects such as Doi-Hopf modules, Yetter-Drinfeld modules and coalgebra Galois extensions  (see, for instance, \cite{Abu}, \cite{BBR},  \cite{Brz2.5}, \cite{Brz3}, \cite{Brz4}, \cite{BCT0}, \cite{BCT}, \cite{CaDe}, \cite{Jia}, \cite{Schbg}).

\smallskip
In \cite{Brz2}, Brzezi\'{n}ski introduced the Hochschild complex $\mathscr C^\bullet(A,C,\psi)$ of an entwining structure and proceeded to construct Gerstenhaber like structures on the cohomology groups. The starting point of this paper was to find a corresponding cyclic cohomology  theory $H^\bullet_\lambda(A,C,\psi)$ for an entwining structure. We then study the cocycles and coboundaries in this theory using differential graded algebras in a manner similar to Connes \cite{C1}, \cite{C2}.

\smallskip Similar to the classical approach of Connes \cite{C2}, we take our ``cyclic complex'' $\mathscr C^\bullet_\lambda(A,C,\psi)$ to be a certain subcomplex of the Hochschild complex $\mathscr C^\bullet(A,C,\psi)$ of Brzezi\'{n}ski \cite{Brz2}. In \cite{C2}, Connes showed that the  cocycles and coboundaries in the cyclic cohomology of an algebra $A$ can be described using traces on differential graded algebras over $A$. Accordingly, we show that the cocycles $Z_\lambda^\bullet(A,C,\psi)$ in our theory can be expressed as characters of ``entwined traces'' applied to dg-entwining structures over $(A,C,\psi)$. We also obtain a description of the coboundaries $B^\bullet_\lambda (A,C,\psi)$ in terms of characters of ``vanishing cycles'' over  $(A,C,\psi)$. These descriptions are then applied to construct a pairing
\begin{equation*}
H^m_\lambda(A,C,\psi) \otimes H^n_\lambda(A',C',\psi') \longrightarrow H^{m+n}_\lambda(A \otimes A', C \otimes C', \psi \otimes \psi')
\end{equation*} on cyclic cohomology groups. We mention here that we have previously studied in \cite{BBN} a modified version of the Hochschild theory
of Brzezi\'{n}ski \cite{Brz2} for entwining structures. In the future, we hope to study further the cohomology groups of entwining structures, on the lines of the usual cohomology theories for rings.

\smallskip
We now describe the paper in more detail. For an element $c\otimes a\in C\otimes A$, we will always suppress the summation and write
$\psi(c\otimes a)=a_\psi\otimes c^\psi\in A\otimes C$. We begin in Section 2 by introducing the cyclic complex $\mathscr C_\lambda^\bullet(A,C,\psi)$ of an entwining structure $(A,C,\psi)$. As a vector space, $\mathscr C^n_\lambda(A,C,\psi)$ consists of all $k$-linear maps $g:C\otimes A^{\otimes n+1}\longrightarrow k$ such that
\begin{equation}
g(c\otimes a_1\otimes ... \otimes a_{n+1})= (-1)^ng(c^\psi\otimes a_2\otimes a_3\otimes ... \otimes a_{n+1}\otimes a_{1\psi})
\end{equation} for $c\in C$ and $a_1,...,a_{n+1}\in A$. We show (see Theorem \ref{T2.2}) that $\mathscr C_\lambda^\bullet(A,C,\psi)$ is a subcomplex
of the Hochschild complex of Brzezi\'{n}ski \cite{Brz2}.  We denote by $H^\bullet_\lambda(A,C,\psi)$ the cohomology groups of $\mathscr C_\lambda^\bullet(A,C,\psi)$.

\smallskip
In Section 3, we consider dg-entwining structures over $(A,C,\psi)$. A dg-entwining structure $((R^\bullet,D^\bullet),C,\Psi^\bullet)$ consists of a (not necessarily unital) dg-algebra $(R^\bullet,D^\bullet)$ and an entwining $\Psi^\bullet:C\otimes R^\bullet\longrightarrow R^\bullet\otimes C$ that is a morphism of complexes. Along with an algebra morphism $\rho: A\longrightarrow R^0$ that is compatible with the respective entwinings $\psi:C\otimes A\longrightarrow A\otimes C$ and $\Psi^0:C\otimes R^0\longrightarrow R^0\otimes C$, we say that
 $((R^\bullet,D^\bullet),C,\Psi^\bullet)$ is a dg-entwining over $(A,C,\psi)$. In particular, we show that $\psi:C\otimes A\longrightarrow A\otimes C$ may be extended to produce a dg-entwining structure  $((\Omega^\bullet A,d^\bullet),C,\hat\psi)$ over $(A,C,\psi)$, where $(\Omega^\bullet A,d^\bullet)$
 is the universal differential graded algebra associated to $A$. Further, we show (see Theorem \ref{T3.4}) that  $((\Omega^\bullet A,d^\bullet),C,\hat\psi)$  is universal among dg-entwining structures over $(A,C,\psi)$. 
 
 \smallskip
Let $((R^\bullet,D^\bullet),C,\Psi^\bullet)$ be a dg-entwining structure over $(A,C,\psi)$. By an $n$-dimensional closed graded entwined trace on $((R^\bullet,D^\bullet),C,\Psi^\bullet)$, we will mean a linear map $T:C\otimes R^n\longrightarrow k$ which satisfies
\begin{equation}
T(c\otimes D(r))=0\qquad T(c\otimes r'r'')=(-1)^{ij}T(c^\Psi\otimes r''r'_\Psi)
\end{equation} for all $c\in C$, $r\in R^{n-1}$ and $r'\in R^i$, $r''\in R^j$ such that $i+j=n$.  Together, the datum $((R^\bullet,D^\bullet),C,\Psi^\bullet,\rho,T)$ will be referred to as an $n$-dimensional entwined cycle over $(A,C,\psi)$.   In Theorem \ref{T4.4}, we show that each cyclic cocycle
$g\in Z^n_\lambda(A,C,\psi)$ may be expressed as the character of an $n$-dimensional entwined cycle over $(A,C,\psi)$. 

\smallskip
Let $M_r(A)$ be the ring of $(r\times r)$-matrices with entries in $A$.  Then $\psi:C\otimes A\longrightarrow A\otimes C$ extends in an obvious manner
to an entwining
\begin{equation}
C\otimes M_r(A)=C\otimes (A\otimes M_r(k))\longrightarrow (A\otimes M_r(k))\otimes C =M_r(A)\otimes C
\end{equation}
that we continue to denote by $\psi$.  In Section 5, we show Morita invariance for Hochschild cohomology groups $HH^\bullet(A,C,\psi)$ of matrix rings.  For this, we  show that the  morphisms on the Hochschild complex induced by the inclusion
$inc_1:A\longrightarrow M_r(A)$ and the generalized trace $tr:M_r(A)^{\otimes n+1}\longrightarrow A^{\otimes n+1}$, $n\geq 0$ are homotopy inverses of each other.  It follows (see Proposition \ref{Prps5.5}) that we have mutually inverse isomorphisms
\begin{equation}
inc_{1}^\bullet: HH^\bullet(M_r(A),C,\psi)\longrightarrow HH^\bullet(A,C,\psi)\qquad tr^\bullet: HH^\bullet(A,C,\psi)\longrightarrow HH^\bullet(M_r(A),C,\psi)
\end{equation}
of Hochschild cohomology groups.

\smallskip
The Morita invariance for cyclic cohomology groups $H^\bullet_\lambda(A,C,\psi)$ of matrix rings is shown in Section 6. For this, we consider the subspace 
$\mathscr I^n(A,C,\psi)\subseteq \mathscr C^n(A,C,\psi)$  consisting of maps $g:C\otimes A^{\otimes n+1}\longrightarrow k$ satisfying 
\begin{equation}
g(c\otimes a_1\otimes ...\otimes a_{n+1})= g(c^{\psi^{n+1}}\otimes a_{1\psi}\otimes a_{2\psi}\otimes ...\otimes a_{n+1\psi})
\end{equation} for $c\in C$ and $a_1$,...,$a_{n+1}\in A$. We check that $\mathscr I^\bullet(A,C,\psi)$ is a subcomplex of $\mathscr C^\bullet(A,C,\psi)$ and that there are induced maps $inc_1^\bullet:\mathscr I^\bullet(M_r(A),C,\psi)\longrightarrow \mathscr I^\bullet(A,C,\psi)$ and $tr^\bullet:
\mathscr I^\bullet(A,C,\psi)\longrightarrow \mathscr I^\bullet(M_r(A),C,\psi)$ which are homotopy inverses of each other.
We also show  that $\mathscr I^\bullet(A,C,\psi)$ is a cocyclic module such that $\mathscr C^\bullet_\lambda(A,C,\psi)$ is the subspace invariant under the action of the cyclic operator on $\mathscr I^\bullet(A,C,\psi)$.  It follows (see Theorem \ref{T6.4}) that we have mutually inverse isomorphisms
\begin{equation}
inc_{1}^\bullet: H^\bullet_\lambda(M_r(A),C,\psi)\longrightarrow H^\bullet_\lambda(A,C,\psi)\qquad tr^\bullet: H^\bullet_\lambda(A,C,\psi)\longrightarrow H^\bullet_\lambda(M_r(A),C,\psi)
\end{equation}
of cyclic cohomology groups.

\smallskip
The main purpose of  Section 7 is to obtain a description for the space $B_\lambda^\bullet(A,C,\psi)$ of coboundaries in $\mathscr C^\bullet_\lambda(A,C,\psi)$. We consider the group $\mathbb U(A)$ of units of $A$ and take the subcollection
\begin{equation}
\mathbb U_\psi(A):=\{x \in \mathbb U(A)~ |~ \psi(c \otimes x)=x \otimes c~\mbox{for every}\textrm{ } c \in C\}
\end{equation} We verify that $\mathbb U_\psi(A)$ is a subgroup of $\mathbb U(A)$. We also show that conjugation by an element
$x\in \mathbb U_\psi(A)$ induces the identity map on cyclic cohomology groups $H^\bullet_\lambda(A,C,\psi)$. Using the Morita invariance established
in Section 6, we now obtain a set of sufficient conditions for the cyclic cohomology of an entwining structure to be zero. Accordingly (see Definition \ref{vancyc}), an $n$-dimensional 
entwined cycle $((R^\bullet,D^\bullet),C,\Psi^\bullet,\rho,T)$ is said to be vanishing if $(R^0,C,\Psi^0)$ satisfies these conditions. 

\smallskip We now take $k=\mathbb C$.  In Theorem \ref{T7.6}, we show that a cocycle $g\in Z_\lambda^n(A,C,\psi)$ is a coboundary if and only if it is the character of an $n$-dimensional entwined vanishing cycle
over $(A,C,\psi)$. In particular, the entwined vanishing cycle corresponding to a coboundary $g\in B^\bullet_\lambda(A,C,\psi)$ is constructed with the help of a certain algebra $\mathbf C$ of infinite matrices with complex entries used in \cite{C2}. Taken together, Theorem \ref{T4.4} and Theorem \ref{T7.6} provide a complete description of the cocycles and the coboundaries in the cyclic theory of entwined structures, developed in a manner similar to Connes \cite{C2}.  Our final result is Theorem \ref{fnT}, where we apply these descriptions to construct a pairing
\begin{equation}
H^m_\lambda(A,C,\psi) \otimes H^n_\lambda(A',C',\psi') \longrightarrow H^{m+n}_\lambda(A \otimes A', C \otimes C', \psi \otimes \psi')\qquad m,n\geq 0
\end{equation}
where $(A,C,\psi)$ and $(A',C',\psi')$ are entwining structures.

\section{Cyclic cohomology of an entwining structure}

Let $k$ be a field. Throughout this section and the rest of this paper, we let $A$ be a unital algebra over $k$ and let $C$ be a counital coalgebra over $k$. The product on 
$A$ will be denoted by $\theta: A\otimes A\longrightarrow A$. The coproduct $\Delta :C\longrightarrow C\otimes C$ 
will always be expressed using Sweedler notation $\Delta(c)=c_1\otimes c_2$ for any $c\in C$. The counit on $C$ will be denoted by $\varepsilon :C\longrightarrow k$.  For the sake of convenience, we will denote the tensor powers $A^{\otimes n}$ of the algebra $A$ simply by $A^n$.  Similarly, an element of $C\otimes A^{\otimes n}$ will
be denoted simply by $(c,a_1,...,a_n)$.  We now recall the notion of an entwining structure, introduced by Brzezi\'{n}ski and Majid in \cite{Brz1}. 

\begin{defn}\label{Def2.1} Let $k$ be a field. An entwining structure $(A,C,\psi)$ over $k$ consists of a unital $k$-algebra $A$, a counital $k$-coalgebra $C$ and a $k$-linear map $\psi:C\otimes A
\longrightarrow A\otimes C$ satisfying the following conditions

\begin{equation}\label{ent}
\begin{array}{c}
\psi(c\otimes \theta(a\otimes b))=\psi(c\otimes ab)=(ab)_\psi\otimes c^\psi = a_\psi b_\psi\otimes {c^\psi}^\psi=((\theta\otimes id_C)\circ (id_A\otimes \psi)\circ (\psi\otimes id_A))(c\otimes a\otimes b)\\
(id_A\otimes \Delta)(\psi (c\otimes a))=a_\psi\otimes \Delta(c^\psi)={a_{\psi}}_{\psi}\otimes c_1^\psi \otimes c_2^\psi =((\psi\otimes id_C)\circ (id_C\otimes \psi))(\Delta(c)\otimes a) \\
a_\psi\varepsilon (c^\psi)=\varepsilon(c)a \qquad 1_\psi\otimes c^\psi = 1\otimes c\\
\end{array}
\end{equation}
Here, the summation has been suppressed by writing $\psi(c\otimes a)=a_\psi\otimes c^\psi$ for any $c\in C$ and $a\in A$.

\end{defn}

In this paper, if $A'$ is a non-unital algebra, we will still say that $(A',C,\psi:C\otimes A'\longrightarrow A'\otimes C)$ is an entwining structure if it satisfies all the conditions in \eqref{ent} except for the last condition $1_\psi\otimes c^\psi = 1\otimes c$.

\smallskip

Given an entwining structure $(A,C,\psi)$, Brzezi\'{n}ski \cite{Brz2} introduced the Hochschild complex $\mathscr C^\bullet((A,C,\psi);M)$ of $(A,C,\psi)$ with coefficients in an $A$-bimodule $M$.
\begin{equation}\label{brc}
\begin{array}{c}
\mathscr C^n((A,C,\psi);M)=Hom(C\otimes A^{n},M)\qquad \delta^n: Hom(C\otimes A^{n},M)\longrightarrow Hom(C\otimes A^{n+1},M)\\
\begin{array}{ll}
\delta^n(f)(c,a_1,...,a_{n+1}) & = a_{1\psi}\cdot f(c^\psi,a_2,...,a_{n+1})+\underset{i=1}{\overset{n}{\sum}}(-1)^i f(c,a_1,...,a_ia_{i+1},...,a_{n+1})\\
& \textrm{ } \textrm{ }+ (-1)^{n+1} f(c,a_1,...,a_{n})\cdot a_{n+1}\\
\end{array}
\end{array}
\end{equation}
The cohomology of this complex will be denoted by $HH^\bullet((A,C,\psi);M)$. In particular, when $M=A^\ast=Hom(A,k)$ made into an $A$-bimodule as follows
\begin{equation}\label{bimod}
(a\cdot f\cdot a')(a''):=f(a'a''a) \qquad f\in A^*=Hom(A,k)\qquad a,a',a''\in A
\end{equation}
this complex will be denoted by $\mathscr C^\bullet(A,C,\psi)$ and its cohomology groups by $HH^\bullet(A,C,\psi)$. It is immediate that an element 
$f\in \mathscr C^n(A,C,\psi)=Hom(C\otimes A^{n},A^*)$ may also be expressed as a linear map $g:C\otimes A^{n+1}\longrightarrow k$ by setting
\begin{equation}\label{iso1}
g(c,a_1,...,a_{n+1})=f(c,a_1,...,a_{n})(a_{n+1})
\end{equation}
We now define a subspace $\mathscr C^n_\lambda(A,C,\psi)\subseteq \mathscr C^n(A,C,\psi)=Hom(C\otimes A^{n},A^*)$ by taking the collection of all
$f\in Hom(C\otimes A^{n},A^*)$ that satisfy
\begin{equation}\label{lambda} 
 f(c,a_1,...,a_{n})(a_{n+1})=(-1)^nf(c^\psi,a_2,...,a_{n+1})(a_{1\psi}) 
\end{equation} for every $(c,a_1,...,a_{n+1})\in C\otimes A^{n+1}$.  Equivalently, using \eqref{iso1}, the space $\mathscr C^n_\lambda(A,C,\psi)$ may also be described as the collection of all $g\in Hom(C\otimes A^{n+1},k)$
such that
\begin{equation}\label{lambda1}
g(c,a_1,...,a_{n+1})=(-1)^ng(c^\psi,a_{2},a_3,...,a_{n+1},a_{1\psi})
\end{equation}

\begin{Thm}\label{T2.2} Let $(A,C,\psi)$ be an entwining structure. Then, $(\mathscr C^\bullet_\lambda(A,C,\psi),\delta^\bullet)$ is a subcomplex of the Hochschild complex of $(A,C,\psi)$. 

\end{Thm}

\begin{proof} We consider $f\in \mathscr C^n_\lambda(A,C,\psi)$. We need to verify that $\delta^n(f)\in \mathscr C^{n+1}_\lambda(A,C,\psi)$, i.e., 
\begin{equation} 
 \delta^n(f)(c,a_1,...,a_{n+1})(a_{n+2})=(-1)^{n+1}\delta^n(f)(c^\psi,a_2,...,a_{n+2})(a_{1\psi}) 
\end{equation} Using the description of the differential in \eqref{brc} and the $A$-bimodule structure of $A^*$ described in \eqref{bimod}, we see that
\begin{equation}\label{one}
\begin{array}{ll}
\delta^n(f)(c,a_1,...,a_{n+1})(a_{n+2}) & =  f(c^\psi,a_2,...,a_{n+1})(a_{n+2}a_{1\psi})+\underset{i=1}{\overset{n}{\sum}}(-1)^i f(c,a_1,...,a_ia_{i+1},...,a_{n+1})(a_{n+2})\\
& \textrm{ } \textrm{ }+ (-1)^{n+1} f(c,a_1,...,a_{n}) (a_{n+1}a_{n+2})\\ & \\ 
\delta^n(f)(c^\psi,a_2,...,a_{n+2})(a_{1\psi}) & =  f(c^{\psi\psi},a_3,...,a_{n+2})(a_{1\psi}a_{2\psi})+\underset{i=2}{\overset{n+1}{\sum}}(-1)^{i-1} f(c^\psi,a_2,...,a_{i}a_{i+1},...,a_{n+2})(a_{1\psi})\\
& \textrm{ } \textrm{ }+ (-1)^{n+1} f(c^\psi,a_2,...,a_{n+1})(a_{n+2}a_{1\psi})\\
\end{array}
\end{equation} Applying condition \eqref{lambda}, we obtain
\begin{equation}\label{rel1}
\underset{i=2}{\overset{n}{\sum}}(-1)^i f(c,a_1,...,a_ia_{i+1},...,a_{n+1})(a_{n+2})=(-1)^{n+1} \underset{i=2}{\overset{n}{\sum}}(-1)^{i-1} f(c^\psi,a_2,...,a_{i}a_{i+1},...,a_{n+2})(a_{1\psi})
\end{equation} as well as
\begin{equation}\label{rel2} 
 (-1)^{n+1} f(c,a_1,...,a_{n}) (a_{n+1}a_{n+2}) = (-1)^{n+1} (-1)^nf(c^\psi, a_2,...,a_{n+1}a_{n+2})(a_{1\psi})
\end{equation} Finally, using \eqref{lambda} as well as the properties of an entwining structure in \eqref{ent}, we obtain
\begin{equation}\label{rel3}
-f(c,a_1a_2,...,a_{n+1})(a_{n+2})=(-1)^{n+1}f(c^\psi,a_3,...,a_{n+2})((a_1a_2)_\psi)=(-1)^{n+1} f(c^{\psi\psi},a_3,...,a_{n+2})(a_{1\psi}a_{2\psi})
\end{equation} The result is now clear from \eqref{one}, \eqref{rel1}, \eqref{rel2} and \eqref{rel3}. 
\end{proof}

\begin{defn} Let $(A,C,\psi)$ be an entwining structure. Then, we will say that  $(\mathscr C^\bullet_\lambda(A,C,\psi),\delta^\bullet)$ is the cyclic complex
of $(A,C,\psi)$ and the cyclic cohomology groups will be denoted by $H^\bullet_\lambda(A,C,\psi)$. The cocyles and coboundaries in $(\mathscr C^\bullet_\lambda(A,C,\psi),\delta^\bullet)$ 
will be denoted respectively by $Z^\bullet_\lambda(A,C,\psi)$ and $ B^\bullet_\lambda(A,C,\psi)$.

\end{defn}

It is clear that the Hochschild complex  $\mathscr C^\bullet(A,C,\psi)$ may be rewritten with terms $Hom(C\otimes A^{\otimes n+1},k)$. In that case, the Hochschild differential may be expressed as follows
\begin{equation}\label{hochg}
\begin{array}{c}
\mathscr C^n(A,C,\psi)=Hom(C\otimes A^{n+1},k)\qquad \delta^n: Hom(C\otimes A^{n+1},k)\longrightarrow Hom(C\otimes A^{n+2},k)\\
\delta^n(g)(c,a_1,...,a_{n+2}) =  g(c^\psi,a_2,...,a_{n+1},a_{n+2}a_{1\psi})+\underset{i=1}{\overset{n+1}{\sum}}(-1)^ig(c,a_1,...,a_ia_{i+1},...,a_{n+2})\\
\end{array}
\end{equation}

\section{Entwining of the universal differential graded algebra} 

We continue with $(A,C,\psi)$ being an entwining structure over $k$. We begin this section by considering an entwining structure where the algebra is differential graded.

\begin{defn}\label{grent} Let $(R^\bullet,D^\bullet)$ be a differential (non-negatively) graded,  not necessarily unital $k$-algebra and let $C$ be a counital $k$-coalgebra. A dg-entwining structure over $k$ consists of a $k$-linear map 
\begin{equation*}
\Psi^\bullet: C\otimes R^\bullet \longrightarrow R^\bullet\otimes C
\end{equation*}
of degree zero such that

\smallskip
(1) $\Psi^\bullet :( C\otimes R^\bullet,id_C\otimes D^\bullet) \longrightarrow (R^\bullet \otimes C,D^\bullet\otimes id_C)$ is a morphism of complexes, i.e.,
\begin{equation*}
D^n(r_\Psi)\otimes c^\Psi= (D^n\otimes C)(\Psi^n(c\otimes r))=\Psi^{n+1}(c\otimes D^n(r))=D^n(r)_\Psi\otimes c^\Psi
\end{equation*} for $c\in C$, $r\in R^n$. 

\smallskip
(2) The tuple $(R,C,\Psi)$ is an entwining structure. 
\end{defn}

\begin{defn}\label{grent1}  Let $(A,C,\psi)$ be an entwining structure. A dg-entwining structure over $(A,C,\psi)$ consists of a dg-entwining $((R^\bullet,D^\bullet), C,\Psi^\bullet)$ and a  $k$-algebra morphism $\rho: A\longrightarrow R^0$ such that we have a commutative diagram
\begin{equation}\label{3.1cd}
\begin{CD}
C\otimes A @>\psi>> A\otimes C\\
@Vid_C\otimes \rho VV @VV\rho\otimes id_CV\\
C\otimes R^0 @>\Psi^0>> R^0\otimes C\\
\end{CD}
\end{equation}

\end{defn}

Given the $k$-algebra $A$, we now consider the algebra $\tilde A:=A\oplus k $ with multiplication given by
\begin{equation*}
(a+\mu ) \cdot (a'+\nu )=( aa'+\mu a' + \nu a )+ \mu \nu  
\end{equation*} for $a$, $a'\in A$ and scalars $\mu$, $\nu\in k$. It is clear that $\tilde A$ is also a unital algebra, with $1\in k$ being the unit. However, we note that
the canonical inclusion $A\hookrightarrow \tilde A$ of algebras is not necessarily unital.  

\smallskip
We now consider the universal differential graded algebra $(\Omega^\bullet A,d^\bullet)$ associated to $A$ (see \cite[$\S$ II.1]{C2}). As a graded vector space, it is given by setting 
$\Omega^n A=\tilde A\otimes  A^{\otimes n}$ for $n>0$ and $\Omega^0A=A$. For $n>0$, an element in $\tilde A\otimes A^{\otimes n}$ is a linear combination of terms of the form
\begin{equation}
(a_0+\mu)da_1...da_n \qquad a_i\in A, \mu\in k
\end{equation} By abuse of notation, we will use $(a_0+\mu)da_1...da_n$ to denote an element of $\Omega^nA$ even for $n=0$. In this case, it will be understood that $\mu=0$.  The multiplication in $\Omega A$ is determined by 
\begin{equation}\label{eq3.3}
a_0da_1....da_n=(a_0)\cdot (da_1) \cdot ... \cdot (da_n)\qquad (da)\cdot a'=d(aa')-a (da')\qquad da_1....da_n= (da_1) \cdot ... \cdot (da_n)
\end{equation} for $a$, $a'$, $a_0$, ..., $a_n\in A$. More generally, for elements $p_0$, ...., $p_i$, $q_0$, ..., $q_j\in A$ and $\mu$, $\nu\in k$, we have
\begin{equation}\label{omega}
\begin{array}{l}
((p_0+\mu)dp_1...dp_i)\cdot ((q_0+\nu)dq_1...dq_j)\\
= (p_0+\mu)\left(dp_1....dp_{i-1}d(p_iq_0)dq_1...dq_j+\underset{l=1}{\overset{i-1}{\sum}}(-1)^{i-l}dp_1...d(p_lp_{l+1})...dp_idq_0dq_1...dq_j\right)\\
\quad + (-1)^i(p_0+\mu)p_1dp_2...dp_idq_0dq_1...dq_j+\nu (p_0+\mu)dp_1...dp_idq_1...dq_j\\
\end{array}
\end{equation}
The differential on $\Omega A$ is determined by setting
\begin{equation}\label{eq3.4}
d((a_0+\mu)da_1...da_n )=da_0da_1...da_n
\end{equation} We also define a morphism
\begin{equation}\label{eq3.5}
\hat\psi : C\otimes \Omega A \longrightarrow \Omega A\otimes C \qquad \hat\psi(c\otimes ((a_0+\mu)da_1...da_n))=(a_{0\psi}da_{1\psi}...da_{n\psi})\otimes c^{\psi^{n+1}}+\mu ( da_{1\psi}...da_{n\psi})\otimes c^{\psi^{n}}
\end{equation} In particular, we have $(da)_\psi\otimes c^\psi=\hat\psi(c\otimes da)=
da_\psi\otimes c^\psi$. 

\begin{thm}\label{P3.3} Let $(A,C,\psi)$ be an entwining structure. Then, $((\Omega^\bullet A,d^\bullet),C,\hat\psi)$ is a dg-entwining structure over $(A,C,\psi)$. 

\end{thm}

\begin{proof} From \eqref{eq3.4} and \eqref{eq3.5}, it is evident that $\hat\psi$ is a morphism of complexes. It is also clear that
\begin{equation}\label{eq3.6}
\begin{array}{ll}
&\hat\psi(c\otimes (a_0)\cdot (da_1) \cdot ... \cdot (da_n))=\hat\psi(c\otimes a_0da_1....da_n)=(a_{0\psi}da_{1\psi}...da_{n\psi})\otimes c^{\psi^{n+1}}=((a_{0\psi})\cdot (da_{1\psi})...\cdot (da_{n\psi}))\otimes c^{\psi^{n+1}}\\
&\hat\psi(c\otimes  (da_1) \cdot ... \cdot (da_n))=\hat\psi(c\otimes da_1....da_n)=(da_{1\psi}...da_{n\psi})\otimes c^{\psi^{n}}=( (da_{1\psi})...\cdot (da_{n\psi}))\otimes c^{\psi^{n+1}}\\
\end{array}
\end{equation}  for $a_0$, ..., $a_n\in A$. Further, for $a$, $a'\in A$, we have
\begin{equation}\label{eq3.7}
\begin{array}{ll}
\hat\psi(c\otimes ((da)\cdot a')) & =\hat\psi (c\otimes d(aa'))-\hat\psi (c\otimes (a(da')))\\
& = d(aa')_{\psi}\otimes c^\psi - (a_\psi da'_\psi)\otimes c^{\psi\psi} \\
& = d(a_\psi a'_\psi )\otimes c^{\psi\psi} - (a_\psi da'_\psi)\otimes c^{\psi\psi} \\
& =( (da_\psi)\cdot a'_\psi)\otimes c^{\psi\psi}=( (da)_\psi\cdot a'_\psi)\otimes c^{\psi\psi}\\
\end{array}
\end{equation} Together, \eqref{eq3.6} and \eqref{eq3.7} show that $\hat\psi$ is well behaved with respect to the multiplication
on $\Omega A$. The other conditions in \eqref{ent} for  $(\Omega A,C,\hat\psi)$ to be an entwining structure may also be verified by direct computation.  Finally, it is clear that the maps
$\hat\psi^0:C\otimes   A=C\otimes \Omega^0 A\longrightarrow \Omega^0 A\otimes C =  A\otimes C$ and $\psi: C\otimes A\longrightarrow A\otimes C$
are identical. This proves the result.

\end{proof}

\begin{Thm}\label{T3.4}  Let $((R^\bullet,D^\bullet), C,\Psi^\bullet)$
be a dg-entwining structure over $(A,C,\psi)$ consisting of a $k$-algebra homomorphism $\rho: A\longrightarrow R^0$. Then, there is an induced morphism $\hat\rho:(\Omega^\bullet A,d^\bullet)\longrightarrow  (R^\bullet,D^\bullet)$ of dg-algebras such that $\hat\rho |_A=\rho: A\longrightarrow R^0$ and we have a commutative diagram
\begin{equation}\label{3.8cd}
\begin{CD}
C\otimes \Omega A @>\hat\psi >>  \Omega A\otimes C \\
@Vid_C\otimes \hat\rho VV @VV\hat\rho\otimes id_CV \\
C\otimes R @>\Psi>> R\otimes C\\
\end{CD}
\end{equation}
\end{Thm} 
\begin{proof}  From the universal property
of $(\Omega^\bullet A,d^\bullet)$ (see \cite[$\S$ II.1]{C2}), we know that there is a unique morphism  $\hat\rho:(\Omega^\bullet A,d^\bullet)\longrightarrow  (R^\bullet,D^\bullet)$ of dg-algebras such that $\hat\rho |_A=\rho: \Omega^0A=A\longrightarrow R^0$. In particular, $\hat\rho$ is described as follows
\begin{equation}
\hat\rho((a_0+\mu)da_1...da_n)= (\rho(a_0))\cdot (D\rho(a_1)) \cdot ... \cdot (D\rho(a_n)) +\mu  (D\rho(a_1)) \cdot ... \cdot (D
\rho(a_n))
\end{equation}
for $ a_0,...,a_n\in A$ and $n>0$, where the products on the right hand side are taken in $R$. For the sake of convenience, we will suppress the morphism $\rho$ and often write $\rho(a)\in R^0$ simply as $a$ for any $a\in A$. For $c\in C$, we now compute
\begin{equation*}
\begin{array}{ll}
((\hat\rho\otimes id_C)\circ \hat\psi)(c\otimes (a_0+\mu)da_1...da_n) & =(\hat\rho\otimes id_C)((a_{0\psi}da_{1\psi}...da_{n\psi})\otimes c^{\psi^{n+1}}+\mu ( da_{1\psi}...da_{n\psi})\otimes c^{\psi^{n}})\\
&= (a_{0\psi})\cdot ( Da_{1\psi}) \cdot ... \cdot (Da_{n\psi})\otimes c^{\psi^{n+1}}+\mu (Da_{1\psi})\cdot ... \cdot (Da_{n\psi})\otimes c^{\psi^{n}}\\
& \\
(\Psi\circ (id_C\otimes \hat\rho))(c\otimes (a_0+\mu)da_1...da_n) & = \Psi(c\otimes ((a_0)\cdot (Da_1) \cdot ... \cdot (Da_n)) +c\otimes (\mu  (Da_1) \cdot ... \cdot (Da_n)) )\\
& =((a_0)_\Psi\cdot (Da_1)_\Psi \cdot ... \cdot (Da_n)_\Psi)\otimes c^{\Psi^{n+1}}+ (\mu  (Da_1)_\Psi \cdot ... \cdot (Da_n)_\Psi) \otimes c^{\Psi^{n}}\\
& =((a_{0\Psi})\cdot (Da_{1\Psi})  \cdot ... \cdot (Da_{n\Psi}))\otimes c^{\Psi^{n+1}}+ (\mu  (Da_{1\Psi})  \cdot ... \cdot (Da_{n\Psi})) \otimes c^{\Psi^{n}}\\
& =((a_{0\psi})\cdot (Da_{1\psi})  \cdot ... \cdot (Da_{n\psi}))\otimes c^{\psi^{n+1}}+ (\mu  (Da_{1\psi})  \cdot ... \cdot (Da_{n\psi})) \otimes c^{\psi^{n}}\\
\end{array}
\end{equation*} where the replacement  of $\Psi$ by $\psi$ in the last equality follows from the commutativity of \eqref{3.1cd}. We have now shown that the diagram
\eqref{3.8cd} is commutative.
\end{proof}

\section{Entwined traces and classes in cyclic cohomology}

In this section, we will show that cocycles in $(\mathscr C^\bullet_\lambda(A,C,\psi),\delta^\bullet)$  correspond to certain kinds of traces on dg-entwining structures over
$(A,C,\psi)$. We continue to suppress the morphism $\rho: A\longrightarrow R^0$ when working with a dg-entwining structure $((R^\bullet,D^\bullet), C,\Psi^\bullet)$ over 
$(A,C,\psi)$. We begin by introducing the notion of an entwined trace.

\begin{defn}\label{D4.1} Let $((R^\bullet,D^\bullet), C,\Psi^\bullet)$
be a dg-entwining structure. An $n$-dimensional closed graded entwined trace for $((R^\bullet,D^\bullet), C,\Psi^\bullet)$ consists of a linear morphism
\begin{equation}
T: C\otimes R^n\longrightarrow k
\end{equation} satisfying the following conditions

\smallskip
(1) For any $c\in C$ and $r\in R^{n-1}$, we have 
\begin{equation}\label{closed}
T(c\otimes D(r))=0 
\end{equation}

\smallskip
(2) For $r\in R^i$, $r'\in R^j$ with $i+j=n$ and any $c\in C$, we have
\begin{equation}\label{entrace}
T(c\otimes rr')=(-1)^{ij}T(c^\Psi\otimes r' r_\Psi)
\end{equation}

\end{defn}

\begin{defn}\label{D4.2} Let $(A,C,\psi)$ be an entwining structure. An $n$-dimensional entwined cycle over $(A,C,\psi)$ is a tuple
\begin{equation}
((R^\bullet,D^\bullet), C,\Psi^\bullet,T,\rho)
\end{equation} where

\smallskip
(1) $\rho: A\longrightarrow R^0$ is a  morphism of $k$-algebras making $((R^\bullet,D^\bullet), C,\Psi^\bullet)$ a dg-entwining 
structure over $(A,C,\psi)$. 

\smallskip
(2) $T:C\otimes R^n\longrightarrow k$ is an $n$-dimensional  closed graded entwined trace on $((R^\bullet,D^\bullet), C,\Psi^\bullet)$.

\end{defn}

\begin{defn}\label{charc}
Let $(A,C,\psi)$ be an entwining structure and let $((R^\bullet,D^\bullet),C,\Psi^\bullet,T,\rho)$ be an $n$-dimensional entwined cycle over $(A,C,\psi)$. Then, we define the character of the cycle $((R^\bullet,D^\bullet),C,\Psi^\bullet,T,\rho)$ to be
the element $g \in \mathscr C^n(A,C,\psi)$ determined by
\begin{equation*}
g(c \otimes a_1 \otimes \ldots \otimes a_{n+1}):=T\left(c \otimes (\rho(a_1)\cdot D(\rho(a_2))\ldots D(\rho(a_{n+1})))\right)
\end{equation*}
for any $c \otimes a_1 \otimes \ldots \otimes a_{n+1} \in C \otimes A^{n+1}$.
\end{defn}

\begin{thm}\label{P4.3} Let $(A,C,\psi)$ be an entwining structure and let 
$g: C\otimes A^{\otimes n+1}\longrightarrow k$ be a linear morphism. Then, the following are equivalent.

\smallskip
(1) There is an $n$-dimensional entwined cycle $((R^\bullet,D^\bullet), C,\Psi^\bullet,T,\rho)$ over $(A,C,\psi)$ such that
\begin{equation}
g(c,a_0,...,a_{n})=T(c\otimes (\rho(a_0)\cdot D\rho(a_1) \cdot  ... \cdot D\rho(a_n)))\qquad a_i\in A, c\in C
\end{equation}

\smallskip
(2) There exists a  closed graded entwined trace $t:C\otimes \Omega^nA\longrightarrow k$ of dimension $n$ on $((\Omega^\bullet A,d^\bullet),C,\hat\psi)$ such that 
\begin{equation}
g(c,a_0,...,a_{n})=t(c\otimes a_0da_1....da_n)\qquad a_i\in A, c\in C
\end{equation}

\end{thm} 

\begin{proof} (1) $\Rightarrow$ (2) : By Theorem \ref{T3.4}, we obtain a morphism $\hat\rho :(\Omega^\bullet A,d^\bullet)\longrightarrow  (R^\bullet,D^\bullet)$ of dg-algebras 
extending $\rho:A\longrightarrow R^0$. We define $t:C\otimes \Omega^nA\longrightarrow k$ by setting
\begin{equation}
t(c\otimes ( (a_0+\mu)da_1...da_n))=T(c\otimes \hat\rho( (a_0+\mu)da_1...da_n))) \qquad a_i\in A, c\in C
\end{equation} In particular, when $\mu=0$, we get
\begin{equation}
t(c\otimes a_0da_1...da_n)=T(c\otimes (\rho(a_0)\cdot D\rho(a_1) \cdot  ... \cdot D\rho(a_n)))=g(c,a_0,...,a_{n})
\end{equation} We have to verify that $t$ satisfies the conditions \eqref{closed} and \eqref{entrace} in Definition \ref{D4.1}. First, we note that
\begin{equation}
\begin{array}{ll}
t(c\otimes d((a_0+\mu)da_1...da_{n-1}))=t(c\otimes da_0...da_{n-1})&=T(c\otimes (D\rho(a_0) \cdot  ... \cdot D\rho(a_{n-1})))\\
&=T(c\otimes D(\rho(a_0) \cdot  D\rho(a_1) \cdot ... \cdot D\rho(a_{n-1})))=0 \\
\end{array}
\end{equation}
This proves the condition \eqref{closed}. Now, for $\alpha\in \Omega^iA$ and $\alpha'\in \Omega^jA$ with $i+j=n$ and for any $c\in C$, we have
\begin{equation*}
\begin{array}{ll}
t(c\otimes \alpha \cdot \alpha') & = T(c\otimes \hat\rho(\alpha \cdot \alpha'))=T(c\otimes \hat\rho(\alpha)\cdot \hat\rho(\alpha'))=(-1)^{ij} T(c^\Psi\otimes  \hat\rho(\alpha')\cdot \hat\rho(\alpha)_\Psi )\\
&=(-1)^{ij} T(c^{\hat\psi}\otimes  \hat\rho(\alpha') \cdot \hat\rho(\alpha_{\hat\psi}))=(-1)^{ij} T(c^{\hat\psi}\otimes  \hat\rho(\alpha'\cdot \alpha_{\hat\psi}) )\\ 
&=(-1)^{ij} t(c^{\hat\psi}\otimes (\alpha'\cdot \alpha_{\hat\psi}))
\end{array}
\end{equation*} 
where the equality $c^\Psi\otimes  \hat\rho(\alpha') \cdot \hat\rho(\alpha)_\Psi=c^{\hat\psi}\otimes  \hat\rho(\alpha') \cdot \hat\rho(\alpha_{\hat\psi})$
follows from the commutativity of the diagram in \eqref{3.8cd}. This shows that $t$ also satisfies the condition \eqref{entrace}.

\smallskip
(2) $\Rightarrow$ (1) :  From Proposition \ref{P3.3}, we already know that  $((\Omega^\bullet A,d^\bullet),C,\hat\psi)$ is   a dg-entwining structure over $(A,C,\psi)$. Since $t:C\otimes \Omega^nA\longrightarrow k$  is a closed graded entwined
trace of dimension $n$, it follows that  $((\Omega^\bullet A,d^\bullet),C,\hat\psi,t,id_A)$ is an $n$-dimensional entwined cycle over $(A,C,\psi)$. We also have
\begin{equation*}
t(c\otimes  a_0\cdot da_1 \cdot ...\cdot da_n)=t(c\otimes a_0da_1...da_n)=g(c,a_0,...,a_{n})
\end{equation*} for $c\in C$ and $a_i\in A$.
\end{proof}

\begin{Thm}\label{T4.4} Let $(A,C,\psi)$ be an entwining structure and let 
$g: C\otimes A^{\otimes n+1}\longrightarrow k$ be a linear morphism. Then, the following are equivalent.

\smallskip
(1) There is an $n$-dimensional entwined cycle $((R^\bullet,D^\bullet), C,\Psi^\bullet,T,\rho)$ over $(A,C,\psi)$ such that
\begin{equation}
g(c,a_0,...,a_{n})=T(c\otimes (\rho(a_0)\cdot D\rho(a_1) \cdot  ... \cdot D\rho(a_n)))\qquad a_i\in A, c\in C
\end{equation}

\smallskip
(2) There exists a  closed graded entwined trace $t:C\otimes \Omega^nA\longrightarrow k$ of dimension $n$ on $((\Omega^\bullet A,d^\bullet),C,\hat\psi)$ such that 
\begin{equation}
g(c,a_0,...,a_{n})=t(c\otimes a_0da_1....da_n)\qquad a_i\in A, c\in C
\end{equation}

\smallskip
(3) $g\in Z_\lambda^n(A,C,\psi)$. 
\end{Thm} 

\begin{proof} From Proposition \ref{P4.3}, we already know that (1) and (2) are equivalent. 

\smallskip
(3) $\Rightarrow$ (2) : We know that $g\in \mathscr C_\lambda^n(A,C,\psi)$.  We define $t:C\otimes \Omega^nA\longrightarrow k$ by setting
\begin{equation}\label{e4.12}
t(c\otimes (a_0+\mu)da_1....da_n):=g(c,a_0,a_1,...,a_n)
\end{equation} for $c\in C$, $a_i\in A$ and $\mu\in k$. We note that
\begin{equation}\label{eq4.13}
t(c\otimes d((a_0+\mu)da_1....da_{n-1}))=t(c\otimes da_0...da_{n-1})=g(c,0,a_0,...,a_{n-1})=0
\end{equation} We now consider elements $(p_0+\mu)dp_1...dp_i\in \Omega^iA$ and $(q_0+\nu)dq_1...dq_j\in \Omega^jA$ with $i+j=n$. Using the expression for the product
on $\Omega A$ given in \eqref{omega}, we obtain
\begin{equation}\label{e4.14}
\begin{array}{l}
t(c\otimes ((p_0+\mu)dp_1...dp_i)\cdot ((q_0+\nu)dq_1...dq_j))\\
=g(c,p_0,p_1,...,p_{i-1},p_iq_0,q_1,...,q_j)+\underset{l=1}{\overset{i-1}{\sum}}(-1)^{i-l}g(c,p_0,p_1,...,p_lp_{l+1},...p_i,q_0,q_1,...,q_j)\\
\quad +(-1)^i g(c,p_0p_1,p_2,...,p_i,q_0,...,q_j) +(-1)^i\mu g(c,p_1,p_2,...,p_i,q_0,...,q_j) +\nu g(c,p_0,p_1,...,p_i,q_1,...,q_j)\\
\end{array}
\end{equation} On the other hand, we have
\begin{equation}\label{e4.15}
\begin{array}{l}
t(c^{\hat\psi}\otimes ((q_0+\nu)dq_1...dq_j)\cdot ((p_0+\mu)dp_1...dp_i)_{\hat\psi})\\=t(c^{\psi^{i+1}}\otimes (q_0dq_1...dq_j)\cdot (p_{0\psi}dp_{1\psi}...dp_{i\psi}))
+\nu t(c^{\psi^{i+1}}\otimes (dq_1...dq_j)\cdot (p_{0\psi}dp_{1\psi}...dp_{i\psi}))\\ \quad +\mu t(c^{\psi^{i}}\otimes  ((q_0+\nu)dq_1...dq_j)\cdot (dp_{1\psi}...dp_{i\psi}))\\
\end{array}
\end{equation} We need to verify that 
\begin{equation}\label{3>2} t(c\otimes ((p_0+\mu)dp_1...dp_i)\cdot ((q_0+\nu)dq_1...dq_j))=(-1)^{ij}t(c^{\hat\psi}\otimes ((q_0+\nu)dq_1...dq_j)\cdot ((p_0+\mu)dp_1...dp_i)_{\hat\psi})
\end{equation} For this, we compare one by one the terms in \eqref{e4.14} and \eqref{e4.15} using the relations \eqref{e4.12}, \eqref{eq4.13}, the product on $\Omega A$
described in \eqref{omega} and the property of $g\in \mathscr C_\lambda^n(A,C,\psi)$ from \eqref{lambda1}. First, we note that
\begin{equation}\label{e4.17}
\begin{array}{ll}
\mu t(c^{\psi^{i}}\otimes  ((q_0+\nu)dq_1...dq_j)\cdot (dp_{1\psi}...dp_{i\psi}))&=\mu g(c^{\psi^{i}},q_0,q_1,...,q_j,p_{1\psi},...,p_{i\psi})\\
&=(-1)^{ij}(-1)^i\mu g(c,p_1,p_2,...,p_i,q_0,...,q_j)\\
\end{array}
\end{equation} Next, we have
\begin{equation}\label{e4.18}
\begin{array}{ll}
\nu t(c^{\psi^{i+1}}\otimes (dq_1...dq_j)\cdot (p_{0\psi}dp_{1\psi}...dp_{i\psi})) & = (-1)^j\nu g(c^{\psi^{i+1}},q_1,...,q_j,p_{0\psi},...,p_{i\psi})\\
&= (-1)^{ij} \nu g(c,p_0,p_1,...,p_i,q_1,...,q_j)\\
\end{array}
\end{equation} We also have
\begin{equation}\label{e4.19}
\begin{array}{l}
t(c^{\psi^{i+1}}\otimes (q_0dq_1...dq_j)\cdot (p_{0\psi}dp_{1\psi}...dp_{i\psi}))\\
=g(c^{\psi^{i+1}},q_0,...,q_{j-1},q_jp_{0\psi},p_{1\psi},...,p_{i\psi}) + \underset{l=1}{\overset{j-1}{\sum}}(-1)^{j-l}g(c^{\psi^{i+1}},q_0,q_1,...,q_lq_{l+1},...q_j,p_{0\psi},...,p_{i\psi})\\
\quad + (-1)^jg(c^{\psi^{i+1}},q_0q_1,q_2,...,q_j,p_{0\psi},...,p_{i\psi})\\
=(-1)^{ni}g(c^{\psi},p_1,...,p_i,q_0,...,q_{j-1},q_jp_{0\psi}) + \underset{l=1}{\overset{j-1}{\sum}}(-1)^{j-l}(-1)^{n(i+1)}g(c,p_0,p_1,...,p_i,q_0,q_1,...,q_lq_{l+1},...q_j)\\
\quad + (-1)^{j+n(i+1)}g(c,p_0,...,p_i,q_0q_1,q_2,...,q_j)\\
\end{array}
\end{equation} The result of \eqref{3>2} now follows from the fact that $g\in  Z_\lambda^n(A,C,\psi)$ satisfies $\delta(g)=0$, where $\delta$ is the Hochschild differential as described in \eqref{hochg}. 

\smallskip
(2) $\Rightarrow$ (3) : We have  an $n$-dimensional closed graded entwined trace $t:C\otimes \Omega^nA\longrightarrow k$ such that
\begin{equation}
g(c,a_0,...,a_{n})=t(c\otimes a_0da_1....da_n)\qquad a_i\in A, c\in C
\end{equation} Since $t$ is a closed graded entwined trace, we note that  $t(c\otimes da_1...da_n)=
t(c\otimes d(a_1da_2...da_{n}))=0$. Hence, 
\begin{equation}
g(c,a_0,...,a_{n})=t(c\otimes (a_0+\mu)da_1....da_n)\qquad a_i\in A, c\in C, \mu \in k
\end{equation}
To show that $g\in Z_\lambda^n(A,C,\psi)$, we have to verify that
\begin{equation}
\delta(g)(c,p_1,...,p_{n+2})=0 \qquad g(c,p_1,...,p_{n+1})=(-1)^ng(c^\psi,p_2,...,p_{n+1},p_{1\psi})
\end{equation} for any $p_i\in A$ and $c\in C$. Here, $\delta$ denotes the Hochschild differential as described in \eqref{hochg}. We now have
\begin{equation}\label{e4.22}
\begin{array}{ll}
g(c,p_1,...,p_{n+1})-(-1)^ng(c^\psi,p_2,...,p_{n+1},p_{1\psi})&= t(c\otimes p_1dp_2...dp_{n+1})-(-1)^nt(c^\psi\otimes p_2dp_3...dp_{n+1}dp_{1\psi} )\\
&= t(c\otimes p_1dp_2...dp_{n+1})+t(c\otimes (dp_1)(p_2dp_3...dp_{n+1}))\\
&= t(c\otimes d(p_1p_2)dp_3...dp_{n+1})=0\\
\end{array}
\end{equation}
Applying \eqref{e4.22}, we now  see that
\begin{equation}\label{e4.23}
t(c^{\psi^{n+1}}\otimes (p_{n+2})\cdot (p_{1\psi}dp_{2\psi}...dp_{(n+1)\psi}))=(-1)^{n}g(c^\psi,p_2,...,p_{n+1},p_{n+2}p_{1\psi})
\end{equation} Applying \eqref{e4.22}, we can  reverse the arguments in \eqref{e4.14} to see that
\begin{equation}\label{e4.24}
\begin{array}{ll}
t(c\otimes (p_1dp_2...dp_{n+1})(p_{n+2}))&=g(c,p_1,...,p_{n+1}p_{n+2})+\underset{l=1}{\overset{n-1}{\sum}}(-1)^{n-l}g(c,p_1,p_2,...,p_{l+1}p_{l+2},...,p_{n+2})\\
&\quad +(-1)^n g(c,p_1p_2,...,p_{n+1},p_{n+2})\\
\end{array}
\end{equation} From \eqref{e4.23} and \eqref{e4.24} and the fact that $t(c\otimes (p_1dp_2...dp_{n+1})(p_{n+2}))-t(c^{\psi^{n+1}}\otimes (p_{n+2})\cdot (p_{1\psi}dp_{2\psi}...dp_{(n+1)\psi}))=0$, it now follows that $\delta(g)=0$. 
\end{proof}

\section{Morita invariance of Hochschild cohomology}

Let $(A,C,\psi)$ be an entwining structure. We construct a presimplicial module $\mathscr C_\bullet=\mathscr C_\bullet(A,C,\psi)$ as follows:
\begin{equation}\label{presim}
\begin{array}{c}
\{\mathscr C_n(A,C,\psi)=C\otimes A^{\otimes n+1}\}_{n\geq 0}\qquad \{d_i:\mathscr C_n\longrightarrow \mathscr C_{n-1}\}_{0\leq i\leq n}\\ \\
d_i(c,a_1,...,a_{n+1}):=\left\{
\begin{array}{ll}
(c^\psi,a_2,...,a_{n+1}a_{1\psi}) & \mbox{if $i=0$} \\
(c,a_1,...,a_ia_{i+1},...,a_{n+1}) & \mbox{ if $0<i\leq n$} \\
\end{array}\right.
\end{array}
\end{equation}

\begin{lem}\label{L5.1} The collection  $\mathscr C_\bullet=\mathscr C_\bullet(A,C,\psi)$ along with the maps in \eqref{presim} forms a presimplicial module.

\end{lem}

\begin{proof} We need to verify (see, for instance, \cite[$\S$ 1.0.6]{Lod2}) that $d_id_j=d_{j-1}d_i$ for $0\leq i<j\leq n$. This is obvious for $j>1$. Since $j>i\geq 0$, the only remaining case is that of $i=0$ and $j=1$. In that case, we have
\begin{equation}
\begin{array}{ll}
d_0d_1(c,a_1,...,a_{n+1})& = d_0(c,a_1a_2,...,a_{n+1})\\
&= (c^\psi,a_3,...,a_{n+1}(a_1a_2)_\psi)\\&=(c^{\psi\psi},a_3,...,a_{n+1}a_{1\psi}a_{2\psi})\\
&=d_0d_0(c,a_1,...,a_{n+1})\\
\end{array}
\end{equation}
\end{proof}

We continue to denote by $\mathscr C_\bullet(A,C,\psi)$ the complex corresponding to this presimplicial module, equipped with standard differential $\sum_{i=0}^n(-1)^id_i$. The homology groups of this complex will be denoted by $HH_\bullet(A,C,\psi)$. From \eqref{hochg}, it is evident that 
\begin{equation}\label{mmr}
\mathscr C^\bullet(A,C,\psi) =Hom(\mathscr C_\bullet(A,C,\psi),k)
\end{equation} We will now show that the complexes $\mathscr C^\bullet(A,C,\psi)$ and $\mathscr C^\bullet(M_r(A),C,\psi)$ are quasi-isomorphic for any $r\geq 1$, where $M_r(A)$ is the ring of
$(r\times r)$-matrices with entries in $A$. 

\smallskip
First, we extend the entwining $\psi : C\otimes A\longrightarrow A\otimes C$ to a map (still denoted by $\psi$)
\begin{equation}
\psi : C\otimes M_2(A)\longrightarrow M_2(A)\otimes C\qquad \psi(c\otimes (a\otimes E_{ij}(1)))=(a_\psi\otimes E_{ij}(1))\otimes c^\psi \qquad a\in A,c \in C
\end{equation} where $\{E_{ij}(1)\}_{1\leq i,j\leq r}$ is the $(r\times r)$-matrix whose $(i,j)$-th entry is $1$ and all others are $0$. It is immediate that $(M_r(A),C,\psi)$
is an entwining structure.

\smallskip
For any $1\leq p\leq r$, we have a (not necessarily unital) inclusion of rings
\begin{equation}
inc_p:A\longrightarrow M_r(A) \qquad a\mapsto E_{pp}(a)
\end{equation} inducing a morphism of complexes $inc_{p\bullet}:\mathscr C_\bullet(A,C,\psi)\longrightarrow \mathscr C_\bullet(M_r(A),C,\psi)$. On the other hand, consider the generalized trace map (see, for instance, \cite[$\S$ 1.2.1]{Lod2}) 
\begin{equation}\label{gtr1}
tr:M_r(A)^{\otimes n+1}\longrightarrow A^{\otimes n+1} \qquad tr(X^1,...,X^{n+1})=\sum (X^1_{i_1i_2})\otimes( X^2_{i_2i_3})\otimes ...\otimes (X^{n+1}_{i_{n+1}i_1})
\end{equation} where the sum is taken over all possible tuples $(i_1,...,i_{n+1})$. Writing $M_r(A)=A\otimes M_r(k)$, the generalized trace can be expressed as (see, for instance,
\cite[$\S$ 1.2.2]{Lod2})
\begin{equation}\label{gtr2}
tr(a_1u_1\otimes ... \otimes a_{n+1}u_{n+1})=tr(u_1... u_{n+1})(a_1\otimes ... \otimes a_{n+1})
\end{equation} where $a_i\in A$ and $u_i\in M_r(k)$. The generalized trace can be extended to a map (still denoted by $tr$) as follows
\begin{equation}\label{gtr3}
tr:C\otimes M_r(A)^{\otimes n+1}\longrightarrow C\otimes A^{\otimes n+1} \quad tr(c\otimes a_1u_1\otimes ... \otimes a_{n+1}u_{n+1})=tr(u_1... u_{n+1})(c\otimes a_1\otimes ... \otimes a_{n+1})
\end{equation}

\begin{lem}\label{L5.2} The generalized trace induces a morphism of complexes $tr_\bullet : \mathscr C
_\bullet(M_r(A),C,\psi)\longrightarrow \mathscr C_\bullet(A,C,\psi)$. 

\end{lem}

\begin{proof} It suffces to show that the generalized trace commutes with the face maps $d_i$ of the presimplicial modules. From \eqref{gtr3}, this is obvious for $i>0$. For $i=0$, we have
\begin{equation}
\begin{array}{ll}
d_0\circ tr(c, a_1u_1, ..., a_{n+1}u_{n+1})&=tr(u_1...u_{n+1})(c^\psi, a_2,..., a_{n+1}a_{1\psi})\\
&=tr(u_2...u_{n+1}u_1)(c^\psi, a_2,..., a_{n+1}a_{1\psi}) \\
&=tr\circ d_0(c, a_1u_1, ..., a_{n+1}u_{n+1})\\
\end{array}
\end{equation} where $a_i\in A$ and $u_i\in M_r(k)$. 
\end{proof}

\begin{thm}
The maps $inc_{1\bullet}:\mathscr{C}_\bullet(A,C,\psi) \longrightarrow \mathscr{C}_\bullet(M_r(A),C,\psi)$ and $tr_\bullet:\mathscr{C}_\bullet(M_r(A),C,\psi) \longrightarrow \mathscr{C}_\bullet(A,C,\psi)$ are homotopy inverses to each other.
\end{thm}
\begin{proof}
We have
\begin{align*}
(tr_\bullet \circ inc_{1\bullet})(c \otimes a_1 \otimes \ldots \otimes a_{n+1})=tr_\bullet(c \otimes a_1E_{11}(1) \otimes \ldots \otimes a_{n+1}E_{11}(1))=c \otimes a_1 \otimes \ldots \otimes a_{n+1}
\end{align*}
which shows that $tr_\bullet \circ inc_{1\bullet}=id$. Therefore, it remains to show that $inc_{1\bullet} \circ tr_\bullet \sim id$.

\smallskip
For each $n \geq 0$, we define $k$-linear maps $\{h_i:\mathscr{C}_n(M_r(A),C,\psi) \longrightarrow \mathscr{C}_{n+1}(M_r(A),C,\psi)\}_{0 \leq i \leq n}$ given by
\begin{equation*}
\begin{array}{ll}
h_i(c \otimes a_1u_1 \otimes \ldots \otimes a_{n+1}u_{n+1}):=& c \otimes \sum\limits_{1 \leq k,l,\ldots,p,q,s\leq r}  a_1E_{11}({u_1}_{kl}) \otimes a_2E_{11}({u_2}_{lm})\otimes \ldots \otimes a_iE_{11}({u_i}_{pq})\\ 
& \quad \quad \quad \quad \quad \quad \quad  \otimes 1_AE_{1q}(1) \otimes a_{i+1}u_{i+1} \otimes  \ldots \otimes a_{n+1} E_{s1}({u_{n+1}}_{sk})
 \end{array}
\end{equation*}
 for $1 \leq i \leq n$ and
\begin{equation*}
\begin{array}{ll}
h_0(c \otimes a_1u_1 \otimes \ldots \otimes a_{n+1}u_{n+1}):=c \otimes \sum\limits_{1 \leq k,s \leq r} 1_AE_{1k}(1) \otimes a_1u_1 \otimes ...\otimes a_nu_n\otimes  a_{n+1} E_{s1}({u_{n+1}}_{sk})
\end{array}
\end{equation*}
We will now prove that $h=\sum_{i=0}^n (-1)^i h_i$ is a homotopy between $inc_{1\bullet} \circ tr_\bullet$ and the identity. To do so, we check the following relations (see \cite[$\S$ 1.0.8]{Lod2}):
\begin{equation}\label{relations}
\begin{array}{lll}
d_i h_{j}=h_{j-1}d_i & \text{for}~ i<j\\
d_ih_i=d_ih_{i-1} & \text{for}~ 0< i \leq n\\
d_ih_{j}= h_{j}d_{i-1} & \text{for}~ i>j+1\\
d_0h_0= id& \text{and} \quad d_{n+1}h_n=inc_{1\bullet} \circ tr_\bullet
\end{array}
\end{equation}
We have
\begin{align*}
d_0h_0(c \otimes a_1u_1 \otimes \ldots \otimes a_{n+1}u_{n+1})&=d_0\left( c \otimes \sum\limits_{1 \leq k,s \leq r} 1_AE_{1k}(1) \otimes a_1u_1 \otimes ...\otimes a_nu_n\otimes  a_{n+1} E_{s1}({u_{n+1}}_{sk}) \right)\\
&=c^\psi \otimes a_1u_1 \otimes \ldots \otimes a_nu_n \otimes \sum\limits_{1 \leq k,s \leq r}a_{n+1} E_{s1}({u_{n+1}}_{sk}) \left(1_AE_{1k}(1)\right)_\psi\\
&=c \otimes a_1u_1 \otimes \ldots \otimes a_nu_n \otimes \sum\limits_{1 \leq k,s \leq r}a_{n+1} E_{s1}({u_{n+1}}_{sk})E_{1k}(1)\\
&= c \otimes a_1u_1 \otimes \ldots \otimes a_{n+1}u_{n+1}
\end{align*}
The third equality follows from the fact that $\psi(c \otimes 1_AE_{pq}(1))=1_AE_{pq}(1) \otimes c$ for all $1 \leq p,q \leq r$. The fourth equality follows from the fact that $\sum\limits_{1 \leq k,s \leq r}E_{s1}({u_{n+1}}_{sk})E_{1k}(1)=u_{n+1}$. 

\smallskip
Further, using the fact that $E_{1q}(1) E_{s1}({u_{n+1}}_{sk})=0$ unless $q=s$, we have
\begin{align*}
&d_{n+1}h_n(c \otimes a_1u_1 \otimes \ldots \otimes a_{n+1}u_{n+1})\\
&=d_{n+1}\Big(c \otimes \sum\limits_{1 \leq k,l,\ldots,p,q,s\leq r}  a_1E_{11}({u_1}_{kl}) \otimes a_2E_{11}({u_2}_{lm})\otimes \ldots \otimes a_nE_{11}({u_n}_{pq}) \otimes 1_AE_{1q}(1) \otimes  a_{n+1} E_{s1}({u_{n+1}}_{sk})\Big)\\
&=c \otimes \sum\limits_{1 \leq k,l,\ldots,p,q,s\leq r}  a_1E_{11}({u_1}_{kl}) \otimes a_2E_{11}({u_2}_{lm})\otimes \ldots \otimes a_nE_{11}({u_n}_{pq}) \otimes a_{n+1}E_{1q}(1) E_{s1}({u_{n+1}}_{sk})\\
&=c \otimes \sum\limits_{1 \leq k,l,\ldots,p,q\leq r}  a_1E_{11}({u_1}_{kl}) \otimes a_2E_{11}({u_2}_{lm})\otimes \ldots \otimes a_nE_{11}({u_n}_{pq})\otimes a_{n+1}E_{1q}(1) E_{q1}({u_{n+1}}_{qk})\\
&=c \otimes \sum\limits_{1 \leq k,l,\ldots,p,q\leq r}  a_1E_{11}({u_1}_{kl}) \otimes a_2E_{11}({u_2}_{lm})\otimes \ldots \otimes a_nE_{11}({u_n}_{pq}) \otimes a_{n+1}E_{11}({u_{n+1}}_{qk})\\
&=(c \otimes a_1E_{11}(1) \otimes \ldots \otimes a_{n+1}E_{11}(1)) \sum\limits_{1 \leq k,l,\ldots,p,q\leq r} \left( {u_1}_{kl}{u_2}_{lm} \ldots {u_{n+1}}_{qk}\right)\\
&=(c \otimes a_1E_{11}(1) \otimes \ldots \otimes a_{n+1}E_{11}(1)) \sum\limits_{1 \leq k \leq r} ( {u_1}{u_2}\ldots {u_{n+1}})_{kk}\\
&=(c \otimes a_1E_{11}(1) \otimes \ldots \otimes a_{n+1}E_{11}(1))tr(u_1 \ldots u_{n+1})\\
&=(inc_{1\bullet} \circ tr_\bullet)(c \otimes a_1u_1 \otimes \ldots \otimes a_{n+1}u_{n+1})
\end{align*}
Now, for $0<i <j$, we have
\begin{align*}
&d_{i}h_j(c \otimes a_1u_1 \otimes \ldots \otimes a_{n+1}u_{n+1})\\
&=d_i\Big(c \otimes \sum\limits_{1 \leq k,l,\ldots,p,q,s\leq r}  a_1E_{11}({u_1}_{kl}) \otimes a_2E_{11}({u_2}_{lm})\otimes \ldots \otimes a_jE_{11}({u_j}_{pq})  \otimes 1_AE_{1q}(1) \otimes a_{j+1}u_{j+1} \otimes  \ldots \otimes a_{n+1} E_{s1}({u_{n+1}}_{sk})\Big)\\
&=c \otimes \sum\limits_{1 \leq k,l,\ldots,p,q,s\leq r}  a_1E_{11}({u_1}_{kl}) \otimes a_2E_{11}({u_2}_{lm})\otimes \ldots \otimes a_iE_{11}({u_i}_{tt'})a_{i+1}E_{11}({u_{i+1}}_{t'n})\otimes \ldots \otimes a_jE_{11}({u_j}_{pq})\\ 
& \quad \quad \quad \quad \quad \quad \quad  \otimes 1_AE_{1q}(1) \otimes a_{j+1}u_{j+1} \otimes  \ldots \otimes a_{n+1} E_{s1}({u_{n+1}}_{sk})\\
&=c \otimes \sum\limits_{1 \leq k,l,\ldots,p,q,s\leq r}  a_1E_{11}({u_1}_{kl}) \otimes a_2E_{11}({u_2}_{lm})\otimes \ldots \otimes a_ia_{i+1}E_{11}(({u_i}{u_{i+1})}_{tn})\otimes \ldots \otimes a_jE_{11}({u_j}_{pq})\\ 
& \quad \quad \quad \quad \quad \quad \quad  \otimes 1_AE_{1q}(1) \otimes a_{j+1}u_{j+1} \otimes  \ldots \otimes a_{n+1} E_{s1}({u_{n+1}}_{sk})\\
&=h_{j-1}\big(c \otimes  a_1u_1 \otimes a_2u_2\otimes \ldots \otimes a_ia_{i+1}{u_i}u_{i+1}\otimes \ldots \otimes a_{n+1}u_{n+1}\big) \\
&=h_{j-1}d_i(c \otimes a_1u_1 \otimes \ldots \otimes a_{n+1}u_{n+1}))
\end{align*}
Moreover, for $j>0$, 
\begin{align*}
&d_{0}h_j(c \otimes a_1u_1 \otimes \ldots \otimes a_{n+1}u_{n+1})\\
&=d_0\Big(c \otimes \sum\limits_{1 \leq k,l,\ldots,p,q,s\leq r}  a_1E_{11}({u_1}_{kl}) \otimes a_2E_{11}({u_2}_{lm})\otimes \ldots \otimes a_jE_{11}({u_j}_{pq})  \otimes 1_AE_{1q}(1) \otimes a_{j+1}u_{j+1} \otimes  \ldots \otimes a_{n+1} E_{s1}({u_{n+1}}_{sk})\Big)\\
&=c^\psi \otimes \sum\limits_{1 \leq k,l,\ldots,p,q,s\leq r} a_2E_{11}({u_2}_{lm})\otimes \ldots \otimes a_jE_{11}({u_j}_{pq})  \otimes 1_AE_{1q}(1) \otimes a_{j+1}u_{j+1} \otimes  \ldots \otimes a_{n+1} E_{s1}({u_{n+1}}_{sk}) (a_1E_{11}({u_1}_{kl}))_\psi\\
&=c^\psi \otimes \sum\limits_{1 \leq k,l,\ldots,p,q,s\leq r} a_2E_{11}({u_2}_{lm})\otimes \ldots \otimes a_jE_{11}({u_j}_{pq})  \otimes 1_AE_{1q}(1) \otimes a_{j+1}u_{j+1} \otimes  \ldots \otimes a_{n+1}a_{1\psi} E_{s1}({u_{n+1}}_{sk}) E_{11}({u_1}_{kl})\\
&=c^\psi \otimes \sum\limits_{1 \leq k,l,\ldots,p,q,s\leq r} a_2E_{11}({u_2}_{lm})\otimes \ldots \otimes a_jE_{11}({u_j}_{pq})  \otimes 1_AE_{1q}(1) \otimes a_{j+1}u_{j+1} \otimes  \ldots \otimes a_{n+1}a_{1\psi} E_{s1}({(u_{n+1}u_1)}_{sl})\\
&=h_{j-1}(c^\psi \otimes a_2u_2 \otimes \ldots \otimes a_{n+1}a_{1\psi}u_{n+1}u_1)=h_{j-1}d_0(c \otimes a_1u_1 \otimes \ldots \otimes a_{n+1}u_{n+1})
\end{align*}
Using the equality $\sum_{q=1}^rE_{11}({u}_{pq})E_{1q}(1)=E_{1p}(1)u$, we have for $0 < i \leq n$
\begin{align*}
&d_{i}h_i(c \otimes a_1u_1 \otimes \ldots \otimes a_{n+1}u_{n+1})\\
&=d_i\Big(c \otimes \sum\limits_{1 \leq k,l,\ldots,p,q,s\leq r}  a_1E_{11}({u_1}_{kl}) \otimes a_2E_{11}({u_2}_{lm})\otimes \ldots \otimes a_iE_{11}({u_i}_{pq}) \otimes 1_AE_{1q}(1) \otimes a_{i+1}u_{i+1} \otimes  \ldots \otimes a_{n+1} E_{s1}({u_{n+1}}_{sk})\Big)\\
&=c \otimes \sum\limits_{1 \leq k,l,\ldots,p,q,s\leq r}  a_1E_{11}({u_1}_{kl}) \otimes a_2E_{11}({u_2}_{lm})\otimes \ldots \otimes a_iE_{11}({u_i}_{pq})E_{1q}(1) \otimes a_{i+1}u_{i+1} \otimes  \ldots \otimes a_{n+1} E_{s1}({u_{n+1}}_{sk})\\
&=c \otimes \sum\limits_{1 \leq k,l,\ldots,p,s\leq r}  a_1E_{11}({u_1}_{kl}) \otimes a_2E_{11}({u_2}_{lm})\otimes \ldots \otimes a_iE_{1p}(1)u_i \otimes a_{i+1}u_{i+1} \otimes  \ldots \otimes a_{n+1} E_{s1}({u_{n+1}}_{sk})\\
&=d_i\left(c \otimes \sum\limits_{1 \leq k,l,\ldots,p,s\leq r}  a_1E_{11}({u_1}_{kl}) \otimes a_2E_{11}({u_2}_{lm})\otimes \ldots \otimes 1_AE_{1p}(1) \otimes a_iu_i \otimes a_{i+1}u_{i+1} \otimes  \ldots \otimes a_{n+1} E_{s1}({u_{n+1}}_{sk})\right)\\
&=d_ih_{i-1}(c \otimes a_1u_1 \otimes \ldots \otimes a_{n+1}u_{n+1})
\end{align*}

For $i > j+1$, it may be similarly verified that $d_ih_j=h_jd_{i-1}$. This proves the result.
\end{proof}

\begin{Thm}
The morphisms
\begin{equation*}
\begin{array}{ll}
inc_{1\bullet}:HH_\bullet(A,C,\psi) \longrightarrow HH_\bullet(M_r(A),C,\psi)\\
tr_\bullet:HH_\bullet(M_r(A),C,\psi) \longrightarrow HH_\bullet(A,C,\psi)
\end{array}
\end{equation*}
are mutually inverse isomorphisms of Hochschild homologies.
\end{Thm}

\smallskip
For each $n \geq 0$, we now obtain $k$-linear maps $\{h^i:\mathscr C^{n+1}(M_r(A),C,\psi) \longrightarrow \mathscr C^n(M_r(A),C,\psi)\}_{0 \leq i \leq n}$ 
given by $h^i(f)=f \circ h_i$. Explicitly, for $1 \leq i \leq n$, we have
\begin{equation}\label{homco1}
\begin{array}{ll}
(h^i(f))(c \otimes a_1u_1 \otimes a_{n+1}u_{n+1})&=f\Big(c \otimes \sum\limits_{1 \leq k,l,\ldots,p,q,s\leq r}  a_1E_{11}({u_1}_{kl}) \otimes a_2E_{11}({u_2}_{lm})\otimes \ldots \otimes a_iE_{11}({u_i}_{pq})\\ 
& \quad \quad \quad \quad \quad \quad \quad  \otimes 1_AE_{1q}(1) \otimes a_{i+1}u_{i+1} \otimes  \ldots \otimes a_nu_n \otimes a_{n+1} E_{s1}({u_{n+1}}_{sk})\Big)
\end{array}
\end{equation}
and
\begin{equation}\label{homco2}
(h^0(f))(c \otimes a_1u_1 \otimes a_{n+1}u_{n+1})=f\Big(c \otimes \sum\limits_{1 \leq k,s \leq r} 1_AE_{1k}(1) \otimes a_1u_1 \otimes \ldots \otimes a_nu_n \otimes a_{n+1} E_{s1}({u_{n+1}}_{sk})\Big)
\end{equation}

\begin{thm}\label{Prps5.5}
The maps $tr^\bullet:\mathscr{C}^\bullet(A,C,\psi) \longrightarrow \mathscr{C}^\bullet(M_r(A),C,\psi)$ and $inc_1^\bullet:\mathscr{C}^\bullet(M_r(A),C,\psi) \longrightarrow \mathscr{C}^\bullet(A,C,\psi)$ are homotopy inverses to each other. In particular, the morphisms
\begin{equation*}
\begin{array}{ll}
tr^{\bullet}:HH^\bullet(A,C,\psi) \longrightarrow HH^\bullet(M_r(A),C,\psi)\\
inc_1^\bullet:HH^\bullet(M_r(A),C,\psi) \longrightarrow HH^\bullet(A,C,\psi)
\end{array}
\end{equation*}
are mutually inverse isomorphisms of  Hochschild cohomologies.
\end{thm}

\section{Invariant subcomplex and Morita invariance of cyclic cohomology}
Let $(A,C,\psi)$ be an entwining structure. The dual $\mathscr C^\bullet(A,C,\psi)$ of $\mathscr C_\bullet(A,C,\psi)$ is a precosimplicial module, equipped with maps $\{\delta_i:\mathscr C^n(A,C,\psi)\longrightarrow \mathscr C^{n+1}(A,C,\psi)\}_{0\leq i\leq n+1}$. For each $n \geq 0$, we set
$\mathscr I^n(A,C,\psi) \subseteq  \mathscr C^n(A,C,\psi)$ to be the collection of morphisms $g \in Hom(C \otimes A^{n+1},k)$ satisfying
\begin{equation*}
g(c \otimes a_1 \otimes \ldots \otimes a_{n+1})=g(c^{\psi^{n+1}} \otimes a_{1\psi} \otimes \ldots \otimes a_{n+1\psi})
\end{equation*}
for every $c \in C$ and $a_1, \ldots, a_{n+1} \in A$.

\begin{lem}\label{precosim}
Let $(A,C,\psi)$ be an entwining structure. Then, $\mathscr I^\bullet(A,C,\psi)$ is a subcomplex of the Hochschild complex $\mathscr C^\bullet(A,C,\psi)$.
\end{lem}
\begin{proof}
We will show that $\delta_i$ for $0 \leq i \leq n+1$ restricts to $\mathscr I^\bullet(A,C,\psi)$. For any $n \geq 0$ and $g \in \mathscr I^n(A,C,\psi)$, we have
\begin{equation*}
\begin{array}{ll}
(\delta_0g)(c \otimes a_1 \otimes \ldots \otimes a_{n+2})&=g(c^\psi \otimes a_2 \otimes \ldots \otimes a_{n+2}a_{1\psi})\\
&=g\left(c^{\psi^{n+2}} \otimes a_{2\psi} \otimes \ldots \otimes (a_{n+2}a_{1\psi})_\psi \right)\\
&=g\left(c^{\psi^{n+3}} \otimes a_{2\psi} \otimes \ldots \otimes a_{n+2\psi}a_{1\psi\psi}\right)\\
&=(\delta_0g)(c^{\psi^{n+2}} \otimes a_{1\psi} \otimes a_{2\psi} \otimes \ldots \otimes a_{n+2\psi})
\end{array}
\end{equation*}
Hence, $\delta_0g \in \mathscr I^{n+1}(A,C,\psi)$. Moreover, for $1 \leq i \leq n+1$, we have
\begin{equation*}
\begin{array}{ll}
(\delta_ig)(c \otimes a_1 \otimes \ldots \otimes a_{n+2})&=g(c \otimes a_1 \otimes \ldots \otimes a_ia_{i+1} \otimes \ldots \otimes a_{n+2})\\
&=g(c^{\psi^{n+1}} \otimes a_{1\psi} \otimes \ldots \otimes (a_ia_{i+1})_\psi \otimes \ldots \otimes a_{n+2\psi})\\
&=g(c^{\psi^{n+2}} \otimes a_{1\psi} \otimes \ldots \otimes a_{i\psi}a_{i+1\psi} \otimes \ldots \otimes a_{n+2\psi})\\
&=(\delta_ig)(c^{\psi^{n+2}} \otimes a_{1\psi} \otimes a_{2\psi} \otimes \ldots \otimes a_{n+2\psi}) 
\end{array}
\end{equation*}
This shows that $\delta_ig \in \mathscr I^{n+1}(A,C,\psi)$ for each $1 \leq i \leq n+1$. This proves the result.
\end{proof}

We will refer to $\mathscr I^\bullet(A,C,\psi)$ as the invariant subcomplex of $\mathscr C^\bullet(A,C,\psi)$. For each $n \geq 0$, we define the $k$-linear maps $\{\sigma_j:\mathscr I^{n+1}(A,C,\psi) \longrightarrow \mathscr I^{n}(A,C,\psi)\}_{0 \leq j \leq n}$ given by
\begin{equation}\label{deg}
(\sigma_jf)(c \otimes a_1 \otimes \ldots \otimes a_{n+1})=f(c \otimes a_1 \otimes \ldots a_j \otimes 1_A \otimes a_{j+1} \otimes \ldots a_{n+1})
\end{equation} We also define the cyclic operator $\tau_n:\mathscr I^n(A,C,\psi) \longrightarrow \mathscr I^n(A,C,\psi)$ as follows:
\begin{equation}\label{cyclicop}
(\tau_ng)(c \otimes a_1 \otimes \ldots \otimes a_{n+1})=(-1)^ng(c^\psi \otimes a_2 \otimes \ldots \otimes a_{n+1} \otimes a_{1\psi})
\end{equation}

\begin{thm}\label{P6.2}
$\mathscr I^\bullet(A,C,\psi)$ is a cocyclic module.
\end{thm}
\begin{proof}
From the proof of Lemma \ref{precosim}, we know that $\mathscr I^\bullet(A,C,\psi)$ is precosimplicial module. 
Together with the maps in \eqref{deg}, it may be easily verified that $\left(\mathscr I^\bullet(A,C,\psi),\delta_i^\bullet,\sigma_j^\bullet\right)$ is a cosimplicial module.

\smallskip
From \eqref{cyclicop}, we have for $g\in \mathscr I^n(A,C,\psi)$, 
\begin{equation*}
(\tau_n^{n+1}g)(c \otimes a_1 \otimes \ldots \otimes a_{n+1})=(-1)^{n(n+1)}g(c^{\psi^{n+1}} \otimes a_{1\psi} \otimes \ldots \otimes a_{n+1\psi})=g(c \otimes a_1 \otimes \ldots \otimes a_{n+1})
\end{equation*}
It remains therefore to verify the following identities:
\begin{equation*}
\begin{array}{ll}
\delta_i\tau_{n-1}&=-\tau_n\delta_{i-1}\qquad 1 \leq i \leq n\\
\delta_0&=(-1)^n\tau_n\delta_n\\
\sigma_i\tau_{n+1}&=-\tau_n\sigma_{i-1}\qquad 1 \leq i \leq n\\
\sigma_0\tau^2_{n+1}&=(-1)^n\tau_n\sigma_n
\end{array}
\end{equation*}
For $1 < i \leq n$ and $g \in \mathscr I^{n-1}(A,C,\psi)$, we have
\begin{equation*}
\begin{array}{ll}
(\delta_i\tau_{n-1}g)(c \otimes a_1 \otimes \ldots \otimes a_{n+1})&=(\tau_{n-1}g)(c \otimes a_1 \otimes \ldots \otimes a_ia_{i+1} \otimes \ldots a_{n+1})\\
&=(-1)^{n-1}g(c^\psi \otimes a_2 \otimes \ldots \otimes a_ia_{i+1} \otimes \ldots a_{n+1} \otimes a_{1\psi})\\
&=(-1)^{n-1}(\delta_{i-1}g)(c^\psi \otimes a_2 \otimes \ldots \otimes a_i \otimes a_{i+1} \otimes \ldots a_{n+1} \otimes a_{1\psi})\\
&=-(\tau_n\delta_{i-1}g)(c \otimes a_1 \otimes \ldots \otimes a_{n+1})
\end{array}
\end{equation*} For $i=1$, we have
\begin{equation*}
\begin{array}{ll}
(\delta_1\tau_{n-1}g)(c \otimes a_1 \otimes \ldots \otimes a_{n+1})&=(\tau_{n-1}g)(c \otimes a_1a_2 \otimes a_3\otimes \ldots   \ldots a_{n+1})\\
&=(-1)^{n-1}g(c^\psi \otimes a_3 \otimes \ldots \otimes a_{n+1} \otimes  (a_1a_2)_{\psi})\\
&=(-1)^{n-1}(\delta_{0}g)(c^\psi \otimes a_2 \otimes \ldots \otimes a_{n+1} \otimes a_{1\psi})\\
&=-(\tau_n\delta_{0}g)(c \otimes a_1 \otimes \ldots \otimes a_{n+1})
\end{array}
\end{equation*}
Clearly, $\delta_0=(-1)^n\tau_n\delta_n$. It may also be easily verified that $\sigma_i\tau_{n+1}=-\tau_n\sigma_{i-1}$ for $1 \leq i \leq n$.
Further, for any $g' \in \mathscr I^{n+1}(A,C,\psi)$
\begin{equation*}
\begin{array}{ll}
(\sigma_0\tau_{n+1}^2g')(c \otimes a_1 \otimes \ldots \otimes a_{n+1})&=(\tau_{n+1}^2g')(c \otimes 1_A \otimes a_1  \otimes \ldots \otimes a_{n+1})\\
&=(-1)^{n+1}(\tau_{n+1}g')(c \otimes a_1  \otimes \ldots \otimes a_{n+1} \otimes 1_A)\\
&=g'(c^\psi \otimes a_2 \otimes \ldots \otimes a_{n+1} \otimes 1_A \otimes a_{1\psi})\\
&=(\sigma_ng')(c^\psi \otimes a_2 \otimes \ldots \otimes a_{n+1} \otimes  a_{1\psi})\\
&=(-1)^n(\tau_n\sigma_ng')(c \otimes a_1 \otimes \ldots \otimes a_{n+1})
\end{array}
\end{equation*}
This proves the result.
\end{proof}

\begin{thm}\label{P6.3}
(1) The maps $inc_1^\bullet:\mathscr{C}^\bullet(M_r(A),C,\psi) \longrightarrow \mathscr{C}^\bullet(A,C,\psi)$ and $tr^\bullet:\mathscr{C}^\bullet(A,C,\psi) \longrightarrow \mathscr{C}^\bullet(M_r(A),C,\psi)$ restrict to the corresponding invariant subcomplexes. In other words, we have morphisms
\begin{equation*}
inc_1^\bullet:\mathscr{I}^\bullet(M_r(A),C,\psi) \longrightarrow \mathscr{I}^\bullet(A,C,\psi)\qquad 
tr^\bullet:\mathscr{I}^\bullet(A,C,\psi) \longrightarrow \mathscr{I}^\bullet(M_r(A),C,\psi)
\end{equation*}

\smallskip
(2) The morphisms $inc_1^\bullet:\mathscr{I}^\bullet(M_r(A),C,\psi) \longrightarrow \mathscr{I}^\bullet(A,C,\psi)$ and $tr^\bullet:\mathscr{I}^\bullet(A,C,\psi) \longrightarrow \mathscr{I}^\bullet(M_r(A),C,\psi)$ are homotopy inverses of each other. Hence, the invariant subcomplexes $\mathscr{I}^\bullet(M_r(A),C,\psi)$ and $\mathscr{I}^\bullet(A,C,\psi)$ are quasi-isomorphic.
\end{thm}
\begin{proof}
(1) For any $n \geq 0$ and $f \in \mathscr{I}^n(M_r(A),C,\psi)$, we have, for $a_i\in A$, $c\in C$:
\begin{equation*}
\begin{array}{ll}
(inc_1^n(f))(c \otimes a_1 \otimes \ldots \otimes a_{n+1})&=f\left(c \otimes a_1E_{11}(1) \otimes \ldots \otimes a_{n+1}E_{11}(1) \right)\\
&=f\left(c^{\psi^{n+1}} \otimes a_{1\psi}E_{11}(1) \otimes \ldots  \otimes a_{n+1\psi}E_{11}(1) \right)\\
&=(inc_1^n(f))(c^{\psi^{n+1}} \otimes a_{1\psi} \otimes \ldots  \otimes a_{n+1\psi})\\
\end{array}
\end{equation*}
Hence, $inc_1^n(f) \in \mathscr{I}^n(A,C,\psi)$. Further, for any $g \in \mathscr{I}^n(A,C,\psi)$, we have, for $a_i\in A$, $u_i\in M_r(k)$, $c\in C$,
\begin{equation*}
\begin{array}{ll}
(tr^n(g))(c \otimes a_1u_1 \otimes \ldots \otimes a_{n+1}u_{n+1})&=g(c \otimes a_1 \otimes \ldots \otimes a_{n+1})tr(u_1u_2\ldots u_{n+1})\\
&=g(c^{\psi^{n+1}} \otimes a_{1\psi} \otimes \ldots  \otimes a_{n+1\psi})tr(u_1u_2\ldots u_{n+1})\\
&=(tr^n(g))\left(c^{\psi^{n+1}} \otimes a_{1\psi}u_1 \otimes \ldots  \otimes a_{n+1\psi}u_{n+1}\right)
\end{array}
\end{equation*}
Therefore, $tr^n( g) \in \mathscr{I}^n(M_r(A),C,\psi)$.

\smallskip

\smallskip
(2) We will now show that the homotopy $h=\sum_{i=0}^n (-1)^i h^i$ as constructed in \eqref{homco1} and \eqref{homco2} restricts to be a homotopy between the maps $tr^\bullet \circ inc_1^\bullet$ and $id_{\mathscr{I}^n(M_r(A),C,\psi)}$. For any $f \in \mathscr{I}^{n+1}(M_r(A),C,\psi)$ and $1 \leq i \leq n$, we have
\begin{align*}
(h^i(f))(c \otimes a_1u_1 \otimes \ldots \otimes  a_{n+1}u_{n+1})&=f\Big(c \otimes \sum\limits_{1 \leq k,l,\ldots,p,q,s\leq r}  a_1E_{11}({u_1}_{kl}) \otimes a_2E_{11}({u_2}_{lm})\otimes \ldots \otimes a_iE_{11}({u_i}_{pq})\\ 
& \quad \quad \quad \quad \quad \quad \quad  \otimes 1_AE_{1q}(1) \otimes a_{i+1}u_{i+1} \otimes  \ldots \otimes a_nu_n \otimes a_{n+1} E_{s1}({u_{n+1}}_{sk})\Big)\\
&=f\Big(c^{\psi^{n+2}} \otimes \sum\limits_{1 \leq k,l,\ldots,p,q,s\leq r}  a_{1\psi}E_{11}({u_1}_{kl}) \otimes a_{2\psi}E_{11}({u_2}_{lm})\otimes \ldots \otimes a_{i\psi}E_{11}({u_i}_{pq})\\ 
& \quad \quad \quad \quad \quad \quad \quad  \otimes (1_A)_\psi E_{1q}(1) \otimes a_{i+1\psi}u_{i+1} \otimes  \ldots \otimes a_{n\psi}u_n \otimes a_{n+1\psi} E_{s1}({u_{n+1}}_{sk})\Big)\\
&=f\Big(c^{\psi^{n+1}} \otimes \sum\limits_{1 \leq k,l,\ldots,p,q,s\leq r}  a_{1\psi}E_{11}({u_1}_{kl}) \otimes a_{2\psi}E_{11}({u_2}_{lm})\otimes \ldots \otimes a_{i\psi}E_{11}({u_i}_{pq})\\ 
& \quad \quad \quad \quad \quad \quad \quad  \otimes 1_AE_{1q}(1) \otimes a_{i+1\psi}u_{i+1} \otimes  \ldots \otimes a_{n\psi}u_n \otimes a_{n+1\psi} E_{s1}({u_{n+1}}_{sk})\Big)\\
&=(h^i(f))(c^{\psi^{n+1}}  \otimes a_{1\psi}u_1 \otimes \ldots \otimes a_{n\psi}u_n \otimes a_{n+1\psi}u_{n+1})
\end{align*}
and
\begin{align*}
(h^0(f))(c \otimes a_1u_1 \otimes \ldots \otimes  a_{n+1}u_{n+1})&=f\Big(c \otimes \sum\limits_{1 \leq k,s \leq r} 1_AE_{1k}(1) \otimes a_1u_1 \otimes \ldots \otimes a_nu_n \otimes a_{n+1} E_{s1}({u_{n+1}}_{sk})\Big)\\
&=f\big(c^{\psi^{n+1}} \otimes \sum\limits_{1 \leq k,s \leq r} 1_AE_{1k}(1) \otimes a_{1\psi}u_1 \otimes \ldots \otimes a_{n\psi}u_n \otimes a_{n+1\psi} E_{s1}({u_{n+1}}_{sk})\Big)\\
&=(h^0(f))(c^{\psi^{n+1}}  \otimes a_{1\psi}u_1 \otimes \ldots \otimes a_{n\psi}u_n \otimes a_{n+1\psi}u_{n+1})
\end{align*}
This proves the result.
\end{proof}

We now observe that $\mathscr C^\bullet_\lambda(A,C,\psi) \subseteq \mathscr I^\bullet(A,C,\psi)$. For any $g \in \mathscr C^n_\lambda(A,C,\psi)$, we have
\begin{equation}
g(c \otimes a_1 \otimes \ldots \otimes a_{n+1})=(-1)^ng(c^\psi \otimes a_2 \otimes \ldots \otimes a_{n+1} \otimes a_{1\psi})=(-1)^{n(n+1)}g(c^{\psi^{n+1}} \otimes a_{1\psi} \otimes \ldots \otimes a_{n+1\psi})
\end{equation} for every $c\in C$ and $a_1$,...,$a_{n+1}\in A$.
In fact, we observe that
\begin{equation}\label{fixed}
\mathscr C^\bullet_\lambda(A,C,\psi) = \{\mbox{$g\in \mathscr I^\bullet(A,C,\psi) $ $\vert$ $\tau_\bullet g=g$}\}
\end{equation}

\begin{Thm}\label{T6.4}   We have mutually inverse isomorphisms
\begin{equation*}
inc_1^\bullet:H^\bullet_\lambda(M_r(A),C,\psi) \longrightarrow H^\bullet_\lambda(A,C,\psi)\qquad 
tr^\bullet:H^\bullet_\lambda(A,C,\psi) \longrightarrow H^\bullet_\lambda(M_r(A),C,\psi)
\end{equation*}
 of cyclic cohomology groups.
\end{Thm}

\begin{proof}
By Proposition \ref{P6.3}, $inc_1^\bullet:\mathscr{I}^\bullet(M_r(A),C,\psi) \longrightarrow \mathscr{I}^\bullet(A,C,\psi)$ and $tr^\bullet:\mathscr{I}^\bullet(A,C,\psi) \longrightarrow \mathscr{I}^\bullet(M_r(A),C,\psi)$ are obtained by restricting the dual of $inc_{1\bullet}: \mathscr C_\bullet(A,C,\psi)\longrightarrow 
\mathscr C_\bullet(M_r(A),C,\psi)$ and  $tr_\bullet:\mathscr C_\bullet(M_r(A),C,\psi)\longrightarrow \mathscr C_\bullet(A,C,\psi)$ to their respective invariant subcomplexes.  It is easily verified that $inc_1^\bullet$ and $tr^\bullet$ commute with the cyclic operators on  $\mathscr{I}^\bullet(A,C,\psi) $ and $ \mathscr{I}^\bullet(M_r(A),C,\psi)$. This shows that $
inc_1^\bullet:\mathscr{I}^\bullet(M_r(A),C,\psi) \longrightarrow \mathscr{I}^\bullet(A,C,\psi)$ and $
tr^\bullet:\mathscr{I}^\bullet(A,C,\psi) \longrightarrow \mathscr{I}^\bullet(M_r(A),C,\psi)$ induce morphisms of the double complexes computing the cyclic
cohomology of $ \mathscr{I}^\bullet(A,C,\psi)$ and $\mathscr{I}^\bullet(M_r(A),C,\psi)$.

\smallskip
From Proposition \ref{P6.3}, we know that $inc_1^\bullet$ and $tr^\bullet$ induce mutually inverse isomorphisms of the Hochschild cohomologies of
$\mathscr{I}^\bullet(A,C,\psi)$ and $ \mathscr{I}^\bullet(M_r(A),C,\psi)$. Since $
\mathscr C^\bullet_\lambda(A,C,\psi) = \{\mbox{$g\in \mathscr I^\bullet(A,C,\psi) $ $\vert$ $\tau_\bullet g=g$}\}$, the result now follows from the Hochschild to cyclic spectral sequence.
\end{proof}

\section{Vanishing cycles and coboundaries}
Let $(A,C,\psi)$ and $(A',C',\psi')$ be entwining structures. A morphism of entwining structures from $(A,C,\psi)$ to $(A',C',\psi')$ is a pair $(\alpha,\gamma)$, where $\alpha:A \longrightarrow A'$ is a $k$-algebra morphism and $\gamma:C \longrightarrow C'$ is a $k$-coalgebra morphism such that
\begin{equation}\label{entmor}
(\alpha \otimes \gamma) \circ \psi=\psi' \circ (\gamma \otimes \alpha)
\end{equation}

\begin{lem}\label{lem7.1}
Let $(\alpha,\gamma):(A,C,\psi)\longrightarrow (A',C',\psi')$ be a morphism of entwining structures. Then, $(\alpha,\gamma)$ induces a 
morphism of complexes $F^\bullet(\alpha,\gamma):\mathscr C^\bullet (A',C',\psi') \longrightarrow \mathscr C^\bullet (A,C,\psi)$ and  a morphism
\begin{equation*}
F^\bullet(\alpha,\gamma):H^\bullet_\lambda(A',C',\psi') \longrightarrow H^\bullet_\lambda(A,C,\psi)
\end{equation*}
of entwined cyclic cohomologies.
\end{lem}
\begin{proof}
Given $(\alpha,\gamma):(A,C,\psi) \longrightarrow (A',C',\psi')$, we have a morphism
\begin{equation*}
F_n(\alpha,\gamma):=\gamma \otimes \alpha^{n+1}:C \otimes A^{n+1} \longrightarrow C' \otimes A'^{n+1}
\end{equation*}
We denote its dual by $F^n(\alpha,\gamma):Hom(C' \otimes A'^{n+1},k)\longrightarrow Hom(C \otimes A^{n+1},k)$. We first need to show that the following diagram commutes for all $n \geq 0$ and $0 \leq i \leq n+1$:
\begin{equation}
\begin{CD}
Hom(C' \otimes A'^{n+1},k) @>\delta_i >> Hom(C' \otimes A'^{n+2},k) \\
@V F^n(\alpha,\gamma) VV @VV F^{n+1}(\alpha,\gamma) V \\
Hom(C \otimes A^{n+1},k) @>\delta_i>> Hom(C \otimes A^{n+2},k)\\
\end{CD}
\end{equation}
For any $f' \in Hom(C' \otimes A'^{n+1},k)$, we have, for $c\in C$, $a_1,...,a_{n+2}\in A$,
\begin{equation}\label{7.3ps}
\begin{array}{ll}
( F^{n+1}(\alpha,\gamma)\delta_0f)(c \otimes a_1 \otimes \ldots \otimes a_{n+2})&=(\delta_0f)(\gamma(c) \otimes \alpha(a_1) \otimes \ldots \otimes \alpha(a_{n+2}))\\
&=f((\gamma(c))^{\psi'} \otimes \alpha(a_2) \otimes \ldots \otimes \alpha(a_{n+2}) (\alpha(a_1))_{\psi'})\\
&=f( \gamma(c^\psi) \otimes \alpha(a_2) \otimes \ldots \otimes \alpha(a_{n+2}a_{1\psi}) )\\
&=(\delta_0F^n(\alpha,\gamma)f)(c \otimes a_1 \otimes \ldots \otimes a_{n+2})
\end{array}
\end{equation}
The third equality in \eqref{7.3ps} follows by using \eqref{entmor} and the fact that $\alpha$ is an algebra map. Similarly, it may be easily verified  that $F^{n+1}(\alpha,\gamma)\delta_i=\delta_iF^n(\alpha,\gamma)$ for each $1 \leq i \leq n+1$. Finally, if $f'\in \mathscr C^n_\lambda(A',C',\psi')$, we have
\begin{equation}
\begin{array}{ll}
(F^n(\alpha,\gamma)f')(c \otimes a_1 \otimes \ldots \otimes a_{n+1})&=f'(\gamma(c)\otimes \alpha(a_1)\otimes ...\otimes \alpha(a_{n+1}))\\
&= (-1)^nf'(\gamma(c)^{\psi'}\otimes \alpha(a_2)\otimes ...\otimes \alpha(a_{n+1})\otimes \alpha(a_1)_{\psi'})\\
&= (-1)^nf'(\gamma(c^\psi)\otimes \alpha(a_2)\otimes ...\otimes \alpha(a_{n+1})\otimes \alpha(a_{1\psi}))\\
&= (-1)^n(F^n(\alpha,\gamma)f')(c^\psi\otimes a_2\otimes ...\otimes a_{n+1}\otimes a_{1\psi})\\
\end{array}
\end{equation} Hence, $(F^n(\alpha,\gamma)f')\in \mathscr C^n_\lambda(A,C,\psi)$. This proves the result. 
\end{proof}

\begin{rem} The definition in \eqref{entmor} and the proof of Lemma \ref{lem7.1} makes sense even if the $k$-algebra morphism $\alpha:A\longrightarrow A'$ is not unital. 
\end{rem}

Suppose that we have morphisms  $(\alpha_1,\gamma_1):(A,C,\psi) \longrightarrow (A',C',\psi')$ and  $(\alpha_2,\gamma_2):(A',C',\psi') \longrightarrow (A'',C'',\psi'')$ of entwining structures. Then, we note that $F^\bullet(\alpha_2 \circ \alpha_1, \gamma_2 \circ \gamma_1)=F^\bullet(\alpha_1,\gamma_1) \circ  F^\bullet(\alpha_2,\gamma_2)$.

\smallskip
For any algebra $A$, let $\mathbb U(A):=\{\mbox{$x\in A$ $\vert$ $\exists$ $y\in A$ such that $xy=yx=1_A$}\}$ be the group  of units of $A$. Given an entwining structure $(A,C,\psi)$, we set
\begin{equation*}
\mathbb U_\psi(A):=\{x \in \mathbb U(A)~ |~ \psi(c \otimes x)=x \otimes c~\mbox{for every}\textrm{ } c \in C\}
\end{equation*}

\begin{lem}
Let $(A,C,\psi)$ be an entwining structure. Then,  $\mathbb U_\psi(A)$ is a subgroup of $\mathbb U(A)$.
\end{lem}
\begin{proof}
Clearly, $1_A \in \mathbb U_\psi(A)$. Let $x, x' \in \mathbb U_\psi(A)$. Then, for any $c\in C$, we have
\begin{equation*}
\psi(c \otimes xx')=(\theta \otimes id_C)(id_A \otimes \psi)(\psi \otimes id_A)(c \otimes x \otimes x')=xx' \otimes c
\end{equation*}
where $\theta:A\otimes A\longrightarrow A$ is the product on $A$.  Hence, $xx' \in \mathbb U_\psi(A)$.

\smallskip
Now let $x \in \mathbb U_\psi(A)$ and let $y\in \mathbb U(A)$ be its inverse. We will show that $y \in \mathbb U_\psi(A)$. For this, we set $\psi(c \otimes y)=\sum_i y_i \otimes c_i\in A\otimes C$. Then,  we have
\begin{equation*}
1_A \otimes c=\psi(c \otimes xy)=(\theta \otimes id_C)(id_A \otimes \psi)(\psi \otimes id_A)(c \otimes x \otimes y)=\sum_i xy_i \otimes c_i
\end{equation*}
Therefore, $y \otimes c=\sum_i yxy_i \otimes c_i=\sum_i y_i \otimes c_i=\psi(c \otimes y)$. Hence, $y \in \mathbb U_\psi(A)$.
\end{proof}

\begin{lem}\label{lem7.3c}
 Let $(A,C,\psi)$ be an entwining structure and let $x \in \mathbb U_\psi(A)$. Then,

\smallskip
(1) the pair $(\phi_x,id_C):(A,C,\psi)\longrightarrow (A,C,\psi)$
is a morphism of entwining structures, where $\phi_x:A\longrightarrow A$ is the inner automorphism given by $\phi_x(a):=xax^{-1}$ for all $a\in A$.

\smallskip
(2) the pair $(\Phi_x,id_C):(M_2(A),C,\psi)\longrightarrow (M_2(A),C,\psi)$ is a morphism of entwining structures, where  $\Phi_x :M_2(A)\longrightarrow M_2(A)$ is the inner automorphism given by
\begin{equation*}
\Phi_x\begin{pmatrix}
a_{11}&a_{12}\\
a_{21}&a_{22}
\end{pmatrix}=\begin{pmatrix}1 & 0 \\ 0 & x \\ \end{pmatrix}\cdot \begin{pmatrix}
a_{11}&a_{12}\\
a_{21}&a_{22}
\end{pmatrix} \cdot \begin{pmatrix}1 & 0 \\ 0 & x^{-1} \\ \end{pmatrix}
\end{equation*}
\end{lem}
\begin{proof}
{\it (1)} Since $x,x^{-1} \in \mathbb U_\psi(A)$, we have
\begin{equation*}
\left(\psi \circ (id_C \otimes \phi_x)\right)(c \otimes a)=\psi(c \otimes xax^{-1})=xa_\psi x^{-1} \otimes c^\psi=\left((\phi_x \otimes id_C)\circ \psi\right)(c \otimes a)
\end{equation*}
for any $c \in C$ and $a \in A$. This shows that $(\phi_x,id_C)$ is a morphism of entwining structures.

\smallskip
{\it (2)} For any $c \in C $ and $\begin{pmatrix}
a_{11}&a_{12}\\
a_{21}&a_{22}
\end{pmatrix} \in M_2(A)$, we have
\begin{equation*}
\begin{array}{ll}
\left(\psi \circ (id_C \otimes \Phi_x)\right)\left(c \otimes \begin{pmatrix}
a_{11}&a_{12}\\
a_{21}&a_{22}
\end{pmatrix} \right)&=\psi\left(c \otimes  \begin{pmatrix}
a_{11}&a_{12}x^{-1}\\
xa_{21}&xa_{22}x^{-1}
\end{pmatrix}\right)\\\\
&=\begin{pmatrix}
a_{11\psi}&a_{12\psi}x^{-1}\\
xa_{21\psi}&xa_{22\psi}x^{-1}
\end{pmatrix}\otimes c^\psi=\left((\Phi_x \otimes id_C)\circ \psi\right)\left(c \otimes \begin{pmatrix}
a_{11}&a_{12}\\
a_{21}&a_{22}
\end{pmatrix} \right)
\end{array}
\end{equation*}
Hence, $(\Phi_x,id_C)$ is a morphism of entwining structures.
\end{proof}

For $x\in \mathbb U_\psi(A)$, we will always denote by $\phi_x:A\longrightarrow A$ and $\Phi_x :M_2(A)\longrightarrow M_2(A)$ the inner automorphisms described in Lemma \ref{lem7.3c}.

\smallskip
Let $(A,C,\psi)$ and $(A',C',\psi')$ be entwining structures. Then, it may be easily seen that the tuple $(A \otimes A', C \otimes C', \psi \otimes \psi')$ is also an entwining structure with the entwining $\psi \otimes \psi': (C \otimes C') \otimes (A \otimes A') \longrightarrow (A \otimes A') \otimes (C \otimes C')$given by
\begin{equation*}
(\psi \otimes \psi')(c \otimes c' \otimes a \otimes a')=a_\psi \otimes a'_{\psi'} \otimes c^\psi \otimes c'^{\psi'}
\end{equation*}
for any $c \otimes c' \in C \otimes C'$ and $a \otimes a' \in A \otimes A'$. 

\begin{lem}\label{unitten}
Let $(A,C,\psi)$ and $(A',C',\psi')$ be entwining structures. Let $x \in \mathbb U_\psi(A)$ and $x' \in \mathbb U_{\psi'}(A')$. Then, $x \otimes x' \in  \mathbb U_{\psi \otimes \psi'}(A \otimes A')$.
\end{lem}
\begin{proof}
This follows immediately from the definitions.
\end{proof}

 We now show that conjugation by a unit $x\in \mathbb U_\psi(A)$ induces the identity on cyclic cohomology 
of $(A,C,\psi)$.

\begin{thm}\label{prop7.2}
Let $(A,C,\psi)$ be an entwining structure. For any $x\in \mathbb U_\psi(A)$, the pair  $(\phi_x,id_C):(A,C,\psi) \longrightarrow (A,C,\psi)$ induces the identity map on $H^\bullet_\lambda(A,C,\psi)$.
\end{thm}
\begin{proof} Using Lemma \ref{lem7.1}, we have morphisms $inc_1^\bullet=F^\bullet(inc_1,id_C)$, $inc_2^\bullet=F^\bullet(inc_2,id_C):H^\bullet_\lambda(M_r(A),C,\psi) \longrightarrow H^\bullet_\lambda(A,C,\psi)$. 
By Theorem \ref{T6.4}, the maps
\begin{equation*}
inc_1^\bullet:H^\bullet_\lambda(M_r(A),C,\psi) \longrightarrow H^\bullet_\lambda(A,C,\psi)\qquad 
tr^\bullet:H^\bullet_\lambda(A,C,\psi) \longrightarrow H^\bullet_\lambda(M_r(A),C,\psi)
\end{equation*}
are mutually inverse isomorphims.  Therefore,
\begin{equation}\label{eq7.3g}
F^\bullet(inc_2,id_C) \circ {(F^\bullet(inc_1,id_C))}^{-1}=F^\bullet(inc_2,id_C)\circ tr^\bullet= F^\bullet(tr \circ inc_2,id_C)=id_{H^\bullet_\lambda(A,C,\psi)}
\end{equation}
We notice that we have  the following commutative diagram:
\begin{equation*}
\begin{CD}
A @>inc_1>>  M_2(A) @<inc_2<< A\\
@Vid_AVV        @VV \Phi_x V @VV \phi_x V\\
A    @> inc_1>>M_2(A)@<inc_2<< A\\
\end{CD}
\end{equation*}
Using Lemma \ref{lem7.3c}, we know that $(\phi_x,id_C)$ and $(\Phi_x,id_C)$ are morphisms of entwining structures. Therefore, we obtain the following commutative diagram:
\begin{equation*}
\begin{CD}
H^\bullet_\lambda(A,C,\psi) @< F^\bullet(inc_1,id_C)<<H^\bullet_\lambda(M_2(A),C,\psi) @> F^\bullet(inc_2,id_C)>> H^\bullet_\lambda(A,C,\psi)\\
@A{F^\bullet(id_A,id_C)}AA        @AAF^\bullet(\Phi_x,id_C)A @AAF^\bullet(\phi_x,id_C)A\\
H^\bullet_\lambda(A,C,\psi) @< F^\bullet(inc_1,id_C) <<H^\bullet_\lambda(M_2(A),C,\psi) @>F^\bullet(inc_2,id_C) >> H^\bullet_\lambda(A,C,\psi)
\end{CD}
\end{equation*}
Therefore, using \eqref{eq7.3g}, we get
\begin{equation*}
F^\bullet(\phi_x,id_C)=  \left(F^\bullet(inc_2,id_C)\right)\circ F^\bullet(inc_1,id_C)^{-1} \circ F^\bullet(id_A,id_C)\circ \left(F^\bullet(inc_1,id_C)\right) \circ F^\bullet(inc_2,id_C)^{-1}=id_{H^\bullet_\lambda(A,C,\psi)}
\end{equation*}
\end{proof}

\begin{thm}\label{vanish}
Let $(A,C,\psi)$ be an entwining structure. Suppose that there is an algebra morphism $\nu:A \longrightarrow A$ and an element $X\in \mathbb U_\psi(M_2(A))$ such that 

\smallskip

\smallskip
(1) the pair $(\nu,id_C): (A,C,\psi) \longrightarrow (A,C,\psi)$ is a morphism of entwining structures, i.e. $(\nu \otimes id_C) \circ \psi=\psi \circ (id_C \otimes \nu)$

\smallskip
(2) the inner automorphism $\phi_X=X(\_\_)X^{-1}:M_2(A)\longrightarrow M_2(A)$ satisfies  \begin{equation*}
\phi_X\begin{pmatrix}
a&0\\
0&\nu(a)
\end{pmatrix}=\begin{pmatrix}
0&0\\
0&\nu(a)
\end{pmatrix}
\end{equation*}
for all $a\in A$. Then, $H^\bullet_\lambda(A,C,\psi)=0$.
\end{thm}
\begin{proof}
Let $\alpha:A \longrightarrow M_2(A)$ and  $\beta:A \longrightarrow M_2(A)$ be the algebra morphisms defined by
\begin{equation*}
\begin{array}{ll}
\alpha(a)=\begin{pmatrix}
a&0\\
0&\nu(a)
\end{pmatrix}=inc_1(a)+ (inc_2 \circ \nu)(a)\\\\
\beta(a)=\begin{pmatrix}
0&0\\
0&\nu(a)
\end{pmatrix}= (inc_2 \circ \nu)(a)
\end{array}
\end{equation*}
for $a \in A$. Since $(inc_2 \circ \nu) \otimes id_C=(inc_2 \otimes id_C)\circ (\nu \otimes id_C)$ and using the fact that the pairs $(\nu,id_C)$, $(inc_1,id_C)$  and $(inc_2,id_C)$ are morphisms of entwining structures, 
 we  see that the pairs $(\alpha,id_C)$ and $(\beta,id_C)$ are also morphisms of entwining structures from $(A,C,\psi)$ to $(M_2(A),C,\psi)$. 
 
\smallskip 
 Applying Lemma \ref{lem7.1}, we now have morphisms on cohomology groups
\begin{equation*}
\begin{array}{ll}
F^\bullet(\alpha,id_C): H^\bullet_\lambda(M_2(A),C,\psi) \longrightarrow H^\bullet_\lambda(A,C,\psi)\\
F^\bullet(\beta,id_C):H^\bullet_\lambda(M_2(A),C,\psi) \longrightarrow H^\bullet_\lambda(A,C,\psi)
\end{array}
\end{equation*}
Further, using assumption {\it (2)}, we have $\phi_X \circ \alpha=\beta$. Applying Proposition \ref{prop7.2} with  $X\in \mathbb U_\psi(M_2(A))$, we obtain
\begin{equation}\label{eq7.4}
F^\bullet(\beta,id_C)=F^\bullet(\phi_X \circ \alpha,id_C)=F^\bullet(\alpha,id_C) \circ F^\bullet(\phi_X,id_C)=F^\bullet(\alpha,id_C)
\end{equation}

\smallskip
Now let $g \in Z^n_\lambda(A,C,\psi)$. Using the isomorphism of cohomology groups in Theorem \ref{T6.4}, we have $\tilde g:=tr^n(g)=g \circ tr \in Z^n_\lambda(M_2(A),C,\psi)$. Let $[\tilde{g}]$ denote the cohomology class of $\tilde{g}$. Then, using \eqref{eq7.4}, we have $F^n(\alpha,id_C)[\tilde{g}]=F^n(\beta,id_C)[\tilde{g}]$ in $H^\bullet_\lambda(A,C,\psi)$. In other words,
\begin{equation*}
\tilde{g} \circ F_n(\alpha,id_C)-\tilde{g} \circ F_n(\beta,id_C) \in B^n_\lambda(A,C,\psi)
\end{equation*}
so that
\begin{equation*}
\tilde{g} \circ F_n(\alpha,id_C)-\tilde{g} \circ F_n(\beta,id_C)=\tilde{g} \circ F_n(inc_1,id_C)=g \in B^n_\lambda(A,C,\psi)
\end{equation*}
This proves the result.
\end{proof}

For the remainder of this paper, we assume $k=\mathbb C$. Let $\mathbf{C}$ be the algebra of infinite matrices with complex entries $(a_{ij})_{i,j \geq 1}$ (see, \cite[p103]{C2}) satisfying the following two conditions:

\begin{itemize}
\item[(i)] the set $\{a_{ij}~ |~ i,j \geq 1\}$ is finite.
\item[(ii)] the number of non-zero entries in each row or each column is bounded.
\end{itemize}

Let $(A,C,\psi)$ be an entwining structure. Then, $\psi$ extends to an entwining $C \otimes \mathbf{C} \otimes A \longrightarrow  \mathbf{C} \otimes A \otimes C$, which we continue to denote by $\psi$ and determined by $c \otimes U \otimes a \mapsto U \otimes a_\psi \otimes c^\psi$ for any $c \in C$, $U \in \mathbf C$ and $a \in A$. Thus, $(\mathbf{C} \otimes A, C, \mathbb{\psi})$ is also an entwining structure. Similarly, $\psi$ can also be extended to obtain an entwining structure $(M_2(\mathbf{C} \otimes A),C,\psi)$.

\begin{lem}\label{lem7.5}
Let $(A,C,\psi)$ be an entwining structure.. Then, $H^\bullet_\lambda(\mathbf{C} \otimes A, C, \mathbb{\psi})=0$.
\end{lem}
\begin{proof}
We will show that the entwining structure $(\mathbf C \otimes A, C, \psi)$ satisfies the assumptions in Proposition \ref{vanish}. By the result in \cite[p104]{C2}, we know that there exists an algebra morphism $\nu:\mathbf{C} \longrightarrow \mathbf{C}$ and a unit $X \in M_2(\mathbf{C})$ such that the corresponding inner automorphism  $\phi_X:M_2(\mathbf{C}) \longrightarrow M_2(\mathbf{C})$ satisfies
\begin{equation}\label{eq7.7g}
\phi_X\begin{pmatrix}
U&0\\
0&\nu(U)
\end{pmatrix}=\begin{pmatrix}
0&0\\
0&\nu(U)
\end{pmatrix}
\end{equation}
for all $U \in \mathbf{C}$. 

\smallskip
The map $\nu$ extends to an algebra morphism $\nu \otimes id_A:\mathbf{C} \otimes A \longrightarrow \mathbf{C} \otimes A$ such that
\begin{equation*}
(\nu \otimes id_A \otimes id_C) \circ \psi=\psi \circ (id_C \otimes \nu \otimes id_A)
\end{equation*}
Hence, the pair $(\nu \otimes id_A, id_C)$ is a morphism of entwining structures $(\mathbf C \otimes A, C, \psi)\longrightarrow (\mathbf C \otimes A, C, \psi)$. Moreover, under the identification $M_2(\mathbf{C} \otimes A)\cong M_2(\mathbf{C}) \otimes A$, we have the unit $X \otimes 1_A \in M_2(\mathbf{C} \otimes A)$. By definition, $\psi(c \otimes X \otimes 1_A)=X \otimes 1_A \otimes c$. Hence,
$X \otimes 1_A  \in \mathbb{U}_{\psi}(M_2(\mathbf{C}) \otimes A)=\mathbb{U}_{\psi}(M_2(\mathbf{C} \otimes A))$.

\smallskip
Clearly, $\phi_{X \otimes 1_A}=\phi_X \otimes id_A:M_2(\mathbf{C}) \otimes A \longrightarrow M_2(\mathbf{C}) \otimes A$. It then follows using \eqref{eq7.7g}  that 
\begin{equation*}
(\phi_X \otimes id_A)\begin{pmatrix}
U \otimes a&0\\
0&(\nu \otimes id_A)(U \otimes a)
\end{pmatrix}=\begin{pmatrix}
0&0\\
0&(\nu \otimes id_A)(U \otimes a)
\end{pmatrix}
\end{equation*}
for any $U \otimes a \in \mathbf{C} \otimes A$. 
Hence, the entwining structure $(\mathbf{C} \otimes A, C, \psi)$ satisfies the assumptions in Proposition \ref{vanish} and therefore, $H^\bullet_\lambda(\mathbf{C} \otimes A, C, \psi)=0$.
\end{proof}

\begin{defn}\label{vancyc}
Let $(A,C,\psi)$ be an entwining structure and $((R^\bullet,D^\bullet), C,\Psi^\bullet,T,\rho)$ be an $n$-dimensional entwined cycle over $(A,C,\psi)$. Suppose that $R^0$ is a unital $k$-algebra and that $\Psi^0(c\otimes 1_{R^0})=1_{R^0}\otimes c$ for each $c\in C$. Then, we say that the cycle $((R^\bullet,D^\bullet), C,\Psi^\bullet,T,\rho)$ is vanishing if the entwining structure $(R^0,C,\Psi^0)$ satisfies the assumptions in Proposition \ref{vanish}.
\end{defn}

As a consequence of Theorem \ref{T4.4}, we observe that the character of an $n$-dimensional entwined cycle over $(A,C,\psi)$ always lies in $Z^n_\lambda(A,C,\psi)$. We will now describe the coboundaries $B_\lambda^n(A,C,\psi)$.

\begin{Thm}\label{T7.6}
Let $(A,C,\psi)$ be an entwining structure and let $g\in Z_\lambda^n(A,C,\psi)$. Then, the following are equivalent:

\smallskip
(1) $g \in B^n_\lambda(A,C,\psi)$.

\smallskip
(2) $g$ is the character of an $n$-dimensional entwined vanishing cycle over  $(A,C,\psi)$.
\end{Thm}
\begin{proof}
{\it(1)} $\Rightarrow$ {\it (2)}. Let $g \in B^n_\lambda(A,C,\psi)$. Then, $g=\delta(g')$ for some $g' \in \mathscr C^{n-1}_\lambda(A,C,\psi)$. Extending $g'$, we obtain an element $\hat{g}' \in \mathscr C^{n-1}(\mathbf C \otimes A,C,\psi)$ as follows
\begin{equation*}
\hat{g}'\left(c \otimes (U_1 \otimes a_1) \otimes \ldots \otimes (U_n \otimes a_n)\right):=g'\left(c \otimes {(U_1)}_{11}a_1 \otimes \ldots \otimes {(U_n)}_{11}a_n \right)
\end{equation*}
for any $c \in C$ and $(U_1 \otimes a_1) \otimes \ldots \otimes (U_n \otimes a_n) \in (\mathbf C \otimes A)^n$. We have
\begin{equation*}
\begin{array}{ll}
\hat{g}'\left(c^\psi \otimes (U_2 \otimes a_2) \otimes \ldots \otimes (U_n \otimes a_n) \otimes (U_1 \otimes a_1)_\psi\right)&=\hat{g}'\left(c^\psi \otimes (U_2 \otimes a_2) \otimes \ldots \otimes (U_n \otimes a_n) \otimes (U_1 \otimes a_{1\psi})\right)\\
&=g'\left(c^\psi \otimes {(U_2)}_{11} a_2 \otimes \ldots \otimes {(U_n)}_{11}a_n \otimes {(U_1)}_{11}a_{1\psi})\right)\\
&=(-1)^{n-1}g'\left(c \otimes {(U_1)}_{11}a_{1} \otimes {(U_2)}_{11} a_2 \otimes \ldots \otimes {(U_n)}_{11}a_n \right)\\
&=(-1)^{n-1}\hat{g}'\left(c \otimes (U_1 \otimes a_1) \otimes \ldots \otimes (U_n \otimes a_n)\right)
\end{array}
\end{equation*}
Hence, $\hat{g}' \in \mathscr C^{n-1}_\lambda(\mathbf C \otimes A,C,\psi)$.

\smallskip
 We now set $\hat{g}'':=\delta(\hat{g}') \in Z^n_\lambda(\mathbf C \otimes A,C,\psi)$. We also consider the algebra morphism $\rho:A \longrightarrow \mathbf{C} \otimes A$ given by $a \mapsto I \otimes a$, where $I$ is the identity matrix in $\mathbf{C}$.
 By the implication {\it(3)} $\Rightarrow$ {\it(2)} in Theorem \ref{T4.4},  there exists an $n$-dimensional closed graded entwined trace $t$ on the dg-entwining structure $\left((\Omega^\bullet(\mathbf{C} \otimes A),d^\bullet),C,\hat\psi\right)$ satisfying in particular that
 \begin{equation}\label{eq7.6}
 \begin{array}{ll}
t(c \otimes \rho(a_1)d(\rho(a_2))\ldots d(\rho(a_{n+1})))&= \hat{g}''(c \otimes \rho(a_1) \otimes \ldots \otimes \rho(a_{n+1}))\\
&=\hat{g}''(c \otimes I \otimes a_1 \otimes \ldots \otimes I \otimes a_{n+1})\\
&=\delta(\hat{g}')(c \otimes I \otimes a_1 \otimes \ldots \otimes I \otimes a_{n+1})\\
&=\delta(g')(c \otimes a_1 \otimes \ldots \otimes a_{n+1})\\
&=g(c \otimes a_1 \otimes \ldots \otimes a_{n+1})
\end{array}
 \end{equation}
Since $(\rho,id_C):(A,C,\psi) \longrightarrow (\mathbf C \otimes A, C,\psi)$ is a morphism of entwining structures, the tuple $((\Omega^\bullet(\mathbf{C} \otimes A),d^\bullet), C, \hat\psi,t, \rho)$ is also an $n$-dimensional entwined cycle over $(A,C,\psi)$. We notice that $\Omega^0(\mathbf{C} \otimes A)=\mathbf C \otimes A$ is unital and $\hat\psi(c\otimes (I\otimes 1_A))=(I\otimes 1_A)\otimes c$ for each $c\in C$. From the proof of Lemma \ref{lem7.5}, we now know that $((\Omega^\bullet(\mathbf{C} \otimes A),d^\bullet), C, \hat\psi,t, \rho)$ is a vanishing cycle. From \eqref{eq7.6} it is clear that $g \in B^n_\lambda(A,C,\psi)$ is the character of this vanishing cycle.

\smallskip
{\it(2)} $\Rightarrow$ {\it (1)}. Let $g \in Z^n_\lambda(A,C,\psi)\subseteq  \mathscr C^n_\lambda(A,C,\psi)$ be the character of an $n$-dimensional entwined vanishing cycle $((R^\bullet,D^\bullet),C,\Psi^\bullet,T,\rho)$ over the entwining structure $(A,C,\psi)$. Then, by definition, the tuple $(R^0,C,\Psi^0)$ is an entwining structure. We define $f \in \mathscr C^n(R^0,C,\Psi^0)$ by setting
\begin{equation*}
f(c \otimes r_1 \otimes \ldots \otimes r_{n+1}):=T(c \otimes r_1D(r_2) \ldots D(r_{n+1}))
\end{equation*}
for any $c \otimes r_1 \otimes \ldots \otimes r_{n+1} \in C \otimes (R^0)^{\otimes n+1}$. Since $\Psi^\bullet$ is a morphism of complexes and $T$ is a closed graded entwined trace of dimension $n$, we have
\begin{equation*}
\begin{array}{ll}
f(c^{\Psi^0} \otimes r_2 \otimes \ldots \otimes r_{n+1} \otimes r_{1\Psi^0})&=T(c^{\Psi^0} \otimes r_2D(r_3)  \ldots D(r_{n+1}) D(r_{1\Psi^0}))\\
&=T(c^{\Psi^1} \otimes r_2D(r_3)  \ldots D(r_{n+1}) D(r_{1})_{\Psi^1})\\
&=(-1)^{n-1} T(c \otimes D(r_1)r_2D(r_3)  \ldots  D(r_{n+1}))\\
&=(-1)^{n-1}\left(T(c \otimes D(r_1r_2)D(r_3)\ldots D(r_{n+1}))-T(c \otimes r_1D(r_2)D(r_3)  \ldots  D(r_{n+1}))\right)\\
&=(-1)^nT(c \otimes r_1D(r_2)D(r_3)  \ldots  D(r_{n+1})))\\
&=(-1)^nf(c \otimes r_1 \otimes \ldots \otimes r_{n+1})
\end{array}
\end{equation*}
This shows that $f \in \mathscr C^n_\lambda(R^0,C,\Psi^0)$. Moreover, using the implication {\it (1)} $\Rightarrow$ {\it (3)} in Theorem \ref{T4.4}, we see that $f \in Z^n_\lambda(R^0,C,\Psi^0)$. Since $((R^\bullet,D^\bullet),C,\Psi^\bullet,T,\rho)$ is a vanishing cycle, we know that $H^\bullet_\lambda(R^0,C,\Psi^0)=0$. Hence, $f=\delta f'$ for some $f' \in \mathscr C^{n-1}_\lambda(R^0,C,\Psi^0)$.

\smallskip
Since $((R^\bullet,D^\bullet),C,\Psi^\bullet)$ is a dg-entwining structure over $(A,C,\psi)$, we have
\begin{equation}\label{last}
(\rho \otimes id_C)\circ \psi=\Psi^0 \circ (id_C \otimes \rho)
\end{equation}
i.e., the pair $(\rho,id_C):(A,C,\psi) \longrightarrow (R^0,C,\Psi^0)$ is a morphism of entwining structures. As in the proof of Lemma \ref{lem7.1},  it is clear that we have an induced morphism of complexes
\begin{equation*}
F^\bullet(\rho,id_C):\mathscr C^{\bullet}_\lambda(R^0,C,\Psi^0) \longrightarrow \mathscr C^{\bullet}_\lambda(A,C,\psi)
\end{equation*}
We set $f'':=F^{n-1}(\rho,id_C)(f')\in \mathscr C^{n-1}_\lambda(A,C,\psi)$. Then,
\begin{equation*}
f''(c \otimes a_1 \otimes \ldots \otimes a_n)=f'(c \otimes \rho(a_1) \otimes \ldots \otimes \rho(a_n))
\end{equation*}
for any $c \otimes a_1 \otimes \ldots \otimes a_n \in C \otimes A^n$. Since $F^\bullet(\rho,id_C)$ is a morphism of complexes, we must have
\begin{equation*}
(\delta f'')(c \otimes a_1 \otimes \ldots \otimes a_{n+1})=(\delta f')(c \otimes \rho(a_1) \otimes \ldots \otimes \rho(a_{n+1}))
\end{equation*}
for any $c \otimes a_1 \otimes \ldots \otimes a_{n+1} \in C \otimes A^{n+1}$. Thus, we obtain
\begin{equation*}
\begin{array}{ll}
g(c \otimes a_1 \otimes \ldots \otimes a_{n+1})&=T(c \otimes \rho(a_1)D(\rho(a_2)) \ldots D(\rho(a_{n+1})))\\
&=f(c \otimes \rho(a_1) \otimes \ldots \otimes \rho(a_{n+1}))\\
&=(\delta f')(c \otimes \rho(a_1) \otimes \ldots \otimes \rho(a_{n+1}))\\
&=(\delta f'')(c \otimes a_1 \otimes \ldots \otimes a_{n+1})
\end{array}
\end{equation*}
Hence, $g=\delta f''$, where $f'' \in \mathscr C^{n-1}_\lambda(A,C,\psi)$. This proves the result.
\end{proof}

Taken together, Theorem \ref{T4.4} and Theorem \ref{T7.6} give a complete description of the cocycles and coboundaries of the complex $\mathscr C^\bullet_\lambda(A,C,\psi)$ in terms of entwined cycles over $(A,C,\psi)$. We conclude by applying these descriptions to obtain a pairing on cyclic cohomologies.

\begin{Thm}\label{fnT}
Let $(A,C,\psi)$ and $(A',C',\psi')$ be entwining structures. Then, we have a pairing
\begin{equation*}
H^m_\lambda(A,C,\psi) \otimes H^n_\lambda(A',C',\psi') \longrightarrow H^{m+n}_\lambda(A \otimes A', C \otimes C', \psi \otimes \psi')
\end{equation*}
for any $m,n \geq 0$.
\end{Thm}
\begin{proof}
Let $g \in Z_\lambda^m(A,C,\psi)$ and $g' \in Z_\lambda^n(A',C',\psi')$. Then, by Theorem \ref{T4.4}, we know that $g$ (resp. $g'$) is the character of an $m$ (resp. $n$)-dimensional entwined cycle $((R^\bullet,D^\bullet), C,\Psi^\bullet, T,\rho)$ (resp. $((R'^\bullet,D'^\bullet), C',\Psi'^\bullet, T',\rho')$) over $(A,C,\psi)$ (resp. $(A',C',\psi')$). 

\smallskip
We know  that $((R \otimes R')^\bullet, (D \otimes D')^\bullet)$ is a differential graded algebra, where $(R \otimes R')^p=\bigoplus\limits_{i+j=p} R^i \otimes R'^j$ for $p \geq 0$. The multiplication and differential on $R\otimes R'$ are given respectively by
\begin{align*}
&(r_1 \otimes r_1')(r_2 \otimes r_2')=(-1)^{deg(r_1')deg(r_2)}(r_1r_2 \otimes r_1'r_2')\\
&(D \otimes D')(r \otimes r')=D(r) \otimes r' + (-1)^{deg(r)} r \otimes D'(r')
\end{align*}
for homogenous elements $r_1, r_2, r \in R$ and $r_1', r_2', r' \in R'$. 
Using the fact that $\Psi$ is a map of degree zero and both $\Psi$, $\Psi'$ are morphisms of complexes, we have
\begin{equation*}
\begin{array}{ll}
&\left((D \otimes D' \otimes id_{C \otimes C'}) \circ (\Psi \otimes \Psi'
)\right)(c \otimes c' \otimes r \otimes r')\\
&=(D \otimes D' \otimes id_{C \otimes C'})(r_\Psi \otimes r'_{\Psi'} \otimes c^\Psi \otimes c'^{\Psi'} )\\
&=D(r_\Psi) \otimes r'_{\Psi'} \otimes c^\Psi \otimes c'^{\Psi'}+(-1)^{deg(r_\Psi)} r_\Psi \otimes D'(r'_{\Psi'}) \otimes  c^\Psi \otimes c'^{\Psi'}\\
&=D(r)_\Psi \otimes r'_{\Psi'} \otimes c^\Psi \otimes c'^{\Psi'}+(-1)^{deg(r)} r_\Psi \otimes D'(r')_{\Psi'} \otimes  c^\Psi \otimes c'^{\Psi'}\\
&=(\Psi \otimes \Psi')(c \otimes c' \otimes (D \otimes D')(r \otimes r'))
\end{array}
\end{equation*} where $c\otimes c'\in C\otimes C'$ and $r$, $r'$ are homogeneous elements of $R$ and $R'$ respectively.

\smallskip

This shows that  $\Psi \otimes \Psi'$ is a morphism of complexes. Thus, we obtain a dg-entwining structure $((R \otimes R')^\bullet, (D \otimes D')^\bullet), C \otimes C', \Psi \otimes \Psi')$.  It is also clear that  the tuple $\left(((R \otimes R')^\bullet, (D \otimes D')^\bullet), C \otimes C', \Psi \otimes \Psi', \rho \otimes \rho' \right)$ is a dg-entwining structure over $(A \otimes A', C \otimes C', \psi \otimes \psi')$.
We will now construct a closed graded entwined trace of dimension $(m+n)$ on this dg-entwining structure.

\smallskip
For this, we consider the $k$-linear map $T \otimes T': (C \otimes C') \otimes (R \otimes R')^{m+n} \longrightarrow k$ given by
\begin{equation*}
(T \otimes T') \left(\bigoplus_{i+j=m+n}(c \otimes c' \otimes r_i \otimes r'_j)\right):=T(c \otimes r_m)T'(c \otimes r'_n)
\end{equation*}
for any $c \otimes c'  \in C \otimes C'$ and $\bigoplus\limits_{i+j=m+n} r_i \otimes r'_j \in (R \otimes R')^{m+n}$.

It may be easily verified that $T \otimes T'$ satisfies the condition in \eqref{closed}. Further, using the fact that $T$ and $T'$ are graded entwined traces, we have
\begin{equation*}
\begin{array}{ll}
&(T \otimes T')(c \otimes c' \otimes (r \otimes r')(s \otimes s'))\\
&=(-1)^{deg(r')deg(s)}(T \otimes T')(c \otimes c' \otimes rs \otimes r's')\\
&=(-1)^{deg(r')deg(s)}T(c \otimes (rs)_m)T'(c \otimes (r's')_n)\\
&=(-1)^{deg(r')deg(s)}(-1)^{deg(r)deg(s)}(-1)^{deg(r')deg(s')}T(c^\Psi \otimes (sr_\Psi)_m)T'(c'^{\Psi'} \otimes (s'r'_{\Psi'})_n)\\
&=(-1)^{deg(r')deg(s)}(-1)^{deg(r)deg(s)}(-1)^{deg(r')deg(s')}(-1)^{deg(s')deg(r_\Psi)}(T \otimes T')(c^\Psi \otimes c'^{\Psi'} \otimes (s \otimes s')(r_\Psi \otimes r'_{\Psi'}))\\
&=(-1)^{deg(r \otimes r')deg(s \otimes s')}(T \otimes T')\left((c \otimes c')^{\Psi \otimes \Psi'} \otimes (s \otimes s')(r \otimes r')_{\Psi \otimes \Psi'}\right)
\end{array}
\end{equation*}
for any $c \otimes c' \in C \otimes C'$ and homogeneous elements $r, s  \in R$ and $r', s' \in R'$. This shows that $T \otimes T'$ satisfies the condition in \eqref{entrace} Hence, $T\otimes T'$ is an $(m+n)$-dimensional closed graded entwined trace.

\smallskip
Therefore, $\left(((R \otimes R')^\bullet, (D \otimes D')^\bullet), C \otimes C', \Psi \otimes  \Psi', T \otimes T', \rho \otimes \rho'\right)$ is an $(m+n)$-dimensional entwined cycle over the entwining structure $(A \otimes A', C \otimes C', \psi \otimes \psi')$.  Using Theorem \ref{T4.4}, we know that the character of this cycle, denoted by $g \otimes g' $, lies in $Z_\lambda^{m+n}(A \otimes A', C \otimes C', \psi \otimes \psi')$. The association $(g,g') \mapsto g \otimes g'$  gives a pairing
\begin{equation}\label{xi}
\xi:Z_\lambda^m(A,C,\psi) \otimes Z_\lambda^n(A',C',\psi') \longrightarrow Z_\lambda^{m+n}(A \otimes A', C \otimes C', \psi \otimes \psi')
\end{equation} From the equivalence of (1) and (2) in Theorem \ref{T4.4}, it is clear that this pairing does not depend on the choice of the cycles  $((R^\bullet,D^\bullet), C,\Psi^\bullet, T,\rho)$ and $((R'^\bullet,D'^\bullet), C',\Psi'^\bullet, T',\rho')$ determining $g$ and $g'$ respectively.

\smallskip
To induce the pairing on cohomologies, it suffices to show that $\xi$ restricts to a  pairing
\begin{equation*}
B_\lambda^m(A,C,\psi) \otimes Z_\lambda^n(A',C',\psi') \longrightarrow B_\lambda^{m+n}(A \otimes A', C \otimes C', \psi\otimes\psi')
\end{equation*}
Let $g \in B_\lambda^m(A,C,\psi)$. Then, by Theorem \ref{T7.6}, we know that $g$ is the character of an $m$-dimensional entwined vanishing cycle $((R^\bullet,D^\bullet), C,\Psi^\bullet, T,\rho)$ over $(A,C,\psi)$. In particular, by Definition \ref{vancyc}, we know that $R^0$ is unital and $\Psi^0(c\otimes 1_{R^0})=1_{R^0}\otimes c$ for each $c\in C$. Using the implication
(1) $\Rightarrow$ (2) in Theorem \ref{T4.4}, we might  as well assume that $R'^0$ is unital and that $\Psi'^0(c'\otimes 1_{R'^0})=1_{R'^0}\otimes c'$ for each $c'\in C'$. In fact, we might even assume that $(R'^\bullet,D'^\bullet)=
(\Omega^\bullet A',d'^\bullet)$. 

\smallskip
It now suffices to show that the cycle $\left(((R \otimes R')^\bullet, (D \otimes D')^\bullet), C \otimes C', \Psi \otimes \Psi', T \otimes T', \rho \otimes \rho'\right)$ used in \eqref{xi} is a vanishing cycle. In other words, we need to verify that the entwining structure $(R^0 \otimes R'^0, C \otimes C', \Psi^0 \otimes \Psi'^0)$ satisfies the assumptions in Proposition \ref{vanish}.

\smallskip
Since $((R^\bullet,D^\bullet), C,\Psi^\bullet, T,\rho)$ is a vanishing cycle, there exists an algebra morphism $\nu:R^0 \longrightarrow R^0$ and a unit $X \in \mathbb{U}_{\Psi^0}(M_2(R^0))$ satisfying the assumptions in Proposition \ref{vanish}.
Extending $\nu$, we have the algebra morphism  $\nu \otimes id_{R'^0}:R^0 \otimes R'^0 \longrightarrow R^0 \otimes R'^0$. Identifying $M_2(R^0 \otimes R'^0) \cong M_2(R^0) \otimes R'^0$ and using Lemma \ref{unitten}, we have $X \otimes 1_{R'^0} \in \mathbb{U}_{\Psi^0 \otimes \Psi'^0}(M_2(R_0) \otimes R'^0) \cong \mathbb{U}_{\Psi^0 \otimes \Psi'^0}(M_2(R_0 \otimes R'^0))$. Clearly, the pair $(\nu \otimes id_{R'^0},id_C \otimes id_{C'})$ is  a morphism of entwining structures. Identifying $\phi_{X \otimes 1_{R'^0}}=\phi_X \otimes  id_{R'^0}:M_2(R^0 \otimes R'^0) \longrightarrow M_2(R^0 \otimes R'^0)$, we also see that
\begin{equation*}
(\phi_X \otimes id_{R'^0})\begin{pmatrix}
r \otimes r'&0\\
0&(\nu \otimes id_{R'^0})(r \otimes r')
\end{pmatrix}=\begin{pmatrix}
0&0\\
0&(\nu \otimes id_{R'^0})(r \otimes r')
\end{pmatrix}
\end{equation*}
for any $r \otimes r' \in  R^0 \otimes R'^0$. 

\smallskip
Thus, all the assumptions in Proposition \ref{vanish} are satisfied by the entwining structure $(R^0 \otimes R'^0, C \otimes C', \Psi^0 \otimes \Psi'^0)$. Hence, $\left(((R \otimes R')^\bullet, (D \otimes D')^\bullet), C \otimes C', \Psi \otimes \Psi', T \otimes T', \rho \otimes \rho'\right)$ is a vanishing cycle.
 This proves the result.
\end{proof}

\end{document}